\documentclass[11pt,reqno,a4paper]{amsart}

\usepackage{url}
\usepackage[colorlinks=true, linkcolor=blue, citecolor=ForestGreen]{hyperref}
\usepackage{cite}

\usepackage{latexsym,amsmath,amssymb,amsfonts,amsthm}
\usepackage{mathrsfs}
\usepackage{mathtools}
\usepackage{dsfont}
\usepackage{bbm}
\usepackage{soul}     
\usepackage[utf8]{inputenc}

\usepackage[usenames, dvipsnames]{color}
\usepackage[usenames, dvipsnames, cmyk]{xcolor}
\usepackage{graphicx}
\usepackage[scriptsize,hang,raggedright]{subfigure}
\usepackage{multirow}
\usepackage{enumerate}

\usepackage{a4wide}
\usepackage{epsfig,epstopdf}

\numberwithin{equation}{section}

\definecolor{grey}{rgb}{.7,.7,.7}
\definecolor{refkey}{gray}{.45}
\definecolor{labelkey}{gray}{.45}

\newcommand{\xupref}[2]{\hspace{-0.3ex}\stackrel{\eqref{#1}}{#2}} 

\newtheorem{theorem}{Theorem}[section]
\newtheorem{proposition}[theorem]{Proposition}
\newtheorem{lemma}[theorem]{Lemma}
\newtheorem*{mainlemma}{Main Lemma}

\theoremstyle{remark}
\newtheorem{remark}[theorem]{Remark}
\theoremstyle{definition}
\newtheorem{definition}[theorem]{Definition}
\newtheorem{example}[theorem]{Example}

\usepackage{tikz}
\usepackage{pgf,pgfplots}
\usetikzlibrary{calc}
\usetikzlibrary{decorations.markings}
\usetikzlibrary{decorations.pathmorphing}
\usetikzlibrary{decorations.shapes}
\usetikzlibrary{shapes,arrows,shapes.geometric,patterns,fadings}

\newcommand{\e}{\varepsilon}

\newcommand{\N}{\mathbb N}

\newcommand{\R}{\mathbb R}

\newcommand{\s}{\mathbb S}

\newcommand{\dd}{\,\mathrm{d}}
\renewcommand{\setminus}{\backslash}

\newcommand{\defeq}{\coloneqq}

\newcommand{\per}{\mathcal{P}}
\newcommand{\nl}{\mathcal{V}}
\newcommand{\nli}{\mathcal{I}}
\newcommand{\en}{\mathcal{E}_\gamma}
\newcommand{\FZ}{\mathscr{C}}
\newcommand{\pol}{\mathscr{P}_n}
\newcommand{\Hone}{\mathcal{H}^1}
\newcommand{\asymm}[2]{#1\triangle#2}

\newcommand{\ba}{\begin{array}}
\newcommand{\ea}{\end{array}}

\newcommand{\bthm}{\begin{theorem}}
\newcommand{\ethm}{\end{theorem}}
\newcommand{\bprop}{\begin{proposition}}
\newcommand{\eprop}{\end{proposition}}
\newcommand{\blemma}{\begin{lemma}}
\newcommand{\elemma}{\end{lemma}}
\newcommand{\bexmpl}{\begin{example}}
\newcommand{\eexmpl}{\end{example}}

\newcommand{\beqn}{\begin{equation}}
\newcommand{\eeqn}{\end{equation}}
\newcommand{\beqns}{\begin{equation*}}
\newcommand{\eeqns}{\end{equation*}}

\newcommand{\Hn}{\mathcal{H}^{d-1}}

\renewcommand{\leq}{\leqslant}
\renewcommand{\geq}{\geqslant}

\definecolor{mygreen}{rgb}{0.1,0.75,0.2}

\newcounter{myenumi}
\DeclareMathOperator{\diam}{diam}

\title[Minimality of polytopes in anisotropic NLIP]{Minimality of polytopes in a nonlocal anisotropic isoperimetric problem}

\author{Marco Bonacini}
\address[Marco Bonacini]{Department of Mathematics, University of Trento, Italy}
\email{marco.bonacini@unitn.it}

\author{Riccardo Cristoferi}
\address[Riccardo Cristoferi]{Department of Mathematics - IMAPP, Radboud University, Nijmegen, The Netherlands}
\email{riccardo.cristoferi@ru.nl}

\author{Ihsan Topaloglu}
\address[Ihsan Topaloglu]{Department of Mathematics and Applied Mathematics, Virginia Commonwealth University, Richmond, VA, USA}
\email{iatopaloglu@vcu.edu}

\date{\today}                                        

\thanks{This is a post-peer-review, pre-copyedit version of an article published in Nonlinear Analysis. The final
authenticated version is available online at: \url{http://dx.doi.org/10.1016/j.na.2020.112223}.}                                    

\begin{document}

\begin{abstract}
We consider the minimization of an energy functional given by the sum of a crystalline perimeter and a nonlocal interaction of Riesz type, under volume constraint. We show that, in the small mass regime, if the Wulff shape of the anisotropic perimeter has certain symmetry properties, then it is the unique global minimizer of the total energy. In dimension two this applies to convex polygons which are reflection symmetric with respect to the bisectors of the angles. We further prove a rigidity result for the structure of (local) minimizers in two dimensions.
\end{abstract}

\maketitle
\tableofcontents


\section{Introduction}\label{sec:intro}

We consider the nonlocal isoperimetric problem
\beqn \label{eq:min}
\min \Bigl\{ \en(E) \,:\, |E|=1 \Bigr\},
\eeqn
among sets of finite perimeter $E\subset\R^d$ with given volume, where $|\cdot|$ denotes the Lebesgue measure in $\R^d$, and the energy functional $\en$ is defined as
\beqn \label{eq:energy}
\en(E) \defeq \int_{\partial^*E} \psi(\nu_E)\dd\Hn + \gamma\int_E\int_E\frac{1}{|x-y|^\alpha}\dd x\dd y
\eeqn
for $\gamma>0$, $\alpha\in(0,d)$. The first term in the energy functional is the (anisotropic) perimeter of $E$, whereas the second term is a Riesz-type nonlocal interaction energy. In this paper we are interested in surface energies determined by \emph{crystalline} surface tensions $\psi$, whose Wulff shapes (which are the corresponding isoperimetric regions) are given by convex polyhedra. We provide detailed definitions of the two terms in the energy, and their properties, in Section~\ref{sec:main}.

The minimization problem \eqref{eq:min} was recently introduced in \cite{ChNeuTo20} as an extension of the classical liquid drop model of Gamow (see \cite{Ga1930}) to the anisotropic setting. Gamow's model, initially developed to predict the mass defect curve and the shape of atomic nuclei, is described by the minimization problem \eqref{eq:min} in the isotropic setting, i.e., with $\psi$ given by the Euclidean norm. The two terms present in the energy functional $\en(E)$ are in direct competition. The surface energy is minimized by a bounded, convex set - the Wulff shape - whereas the repulsive term prefers to disperse the mass into vanishing components diverging infinitely apart. The parameter of the problem, that is $\gamma$, sets a length scale between these competing forces. As such, the liquid drop model is a paradigm for shape optimization via competing short- and long-range interactions on unbounded spaces and it  has recently generated considerable interest in the calculus of variations community (see e.g. \cite{AlBrChTo2017_3,BoCr14,ChPe2010,FFMMM,FrLi2015,Ju2014,KnMu2014,KnMuNo2016,LuOtto2014,MurZal,RW2014} as well as \cite{ChMuTo2017} for a review).

In the anisotropic liquid drop model \eqref{eq:min} the competition is not only between the attractive and repulsive forces, but also between the anisotropy in the surface energy and the isotropy of the Riesz-like interaction energy. As in the isotropic case, the problem admits a minimizer when $\gamma$ is sufficiently small and fails to have minimizers for large values of $\gamma$, see \cite[Theorem~3.1]{ChNeuTo20}. However, as \cite[Theorem 1.1]{ChNeuTo20} shows, when $\psi$ is smooth its Wulff shape $W_\psi$ is not a critical point of the energy $\en(E)$ for any $\gamma>0$, whereas in the isotropic case the ball is the unique global minimizer for $\gamma>0$ sufficiently small. On the other hand, for crystalline surface tensions whose Wulff shape is given by a square the authors prove that the corresponding Wulff shape is the unique (modulo translations) minimizer for sufficiently small $\gamma$.

This demonstrates a fundamentally interesting situation: the regularity and ellipticity of the surface tension $\psi$ determines whether the isoperimetric set $W_\psi$ could also be a minimizer of the perturbed problem \eqref{eq:min}. The qualitative properties of minimizers when $\psi$ is smooth (hence, not equal to $W_\psi$) have recently been studied in \cite{MiTo} in the asymptotic regime $\gamma\to 0$.

In this paper we prove that, for a wide class of crystalline surface tensions, the corresponding isoperimetric set $W_\psi$ \emph{remains} as the minimizer of the nonlocal isoperimetric problem for small values of $\gamma>0$. Specifically, we prove the following:

	\begin{itemize}\setlength\itemsep{1em}
		\item In any dimensions, if the Wulff shape of $\psi$ enjoys particular symmetry properties (in particular, if it is a regular polytope), then it is, up to translations, the unique solution to \eqref{eq:min} for sufficiently small $\gamma$ (see Theorem~\ref{thm:main1} and Theorem~\ref{thm:main3}); in two dimensions, these polygons are exactly those which are reflection symmetric with respect to the bisectors of all angles.
		\item In two dimensions and for every $\gamma>0$, the boundary of any local minimizer of $\en$ can be decomposed into two parts, one of which is a level set of the potential induced by the interaction energy and the other one is aligned with the sides of the corresponding Wulff shape (see Theorem~\ref{thm:main2}). This rigidity result is an adaptation of \cite[Theorem~14]{FigalliMaggiARMA}, where a similar property was firstly observed and proved for the crystalline perimeter perturbed by a bulk potential energy induced by an external force field.
	\end{itemize}
	\smallskip

Our first result relies on a structure theorem obtained by Figalli and Maggi \cite{FigalliMaggiARMA} in two dimensions, recently extended by Figalli and Zhang \cite{FigZha} to higher dimensions, which states that for $\gamma$ sufficiently small, minimizers of $\en$ are polygons with sides aligned with those of the corresponding Wulff shape, a result which essentially reduces the problem to a finite dimensional one. Then the main point of the proof is to show that, restricting to this class of variations, the symmetry assumptions on the polygon yield a quadratic upper bound on the nonlocal energy difference.

It is worth to notice that, for small $\gamma$, minimizers of \eqref{eq:min} are always obtained by perturbations of the Wulff shape of the surface energy, whose sides are translated parallel to themselves; in our result we exhibit an explicit class of Wulff shapes which remain global minimizers for $\gamma>0$. However, we cannot prove that polygons in this class are exactly those with this global minimality property. In other words, it is an open problem to classify the crystalline anisotropies whose Wulff shapes remain the global minimizers of \eqref{eq:min} for $\gamma>0$ sufficiently small. In turn, this would require to characterize the critical points of the nonlocal energy with respect to the restricted class of variations (see Remark~\ref{rmk:critical}).

Our second theorem shows that this rigid structure of minimizers is not just peculiar to the small $\gamma$ regime, but it characterizes (local) minimizers also for large values of $\gamma$. The proof follows the lines of an analogous result in \cite[Theorem~14]{FigalliMaggiARMA} for the minimization of the sum of the crystalline perimeter and an external potential energy. We point out that the rigidity of minimizers in geometric variational problems involving crystalline surface tensions seems to be a ubiquitous property: a similar phenomenon was observed in \cite{Bo2013} for a thin film model, where it was shown that the flat configuration (that is, a configuration with a flat facet parallel to a facet of the Wulff shape) was always a local minimizer.

\subsection*{Structure of the paper}
The paper is organized as follows. In Section~\ref{sec:main} we introduce the necessary definitions and notations, and state the main results of the paper. In Section~\ref{sec:quadratic} we prove the Main Lemma providing the quadratic upper bound. Section~\ref{sec:proofs} is devoted to the proofs of the main results. Finally, in Section~\ref{sec:3d} we outline possible extensions of our results to higher dimensions.


\section{Definitions and main results}\label{sec:main}

\subsection{The energy functional}
As noted in the introduction the energy $\en$ is the sum of an anisotropic surface energy and a nonlocal interaction energy of Riesz type. We start by defining these two terms separately and detail their properties.

\medskip
Given a convex, positively one-homogeneous function $\psi:\R^d\to[0,\infty)$, strictly positive on $\s^{d-1}$, we define the anisotropic surface energy of a set of finite perimeter $E\subset\R^d$ as
\beqn \label{eq:perimeter}
\per_\psi(E) \defeq  \int_{\partial^*E} \psi(\nu_E)\dd\Hn,
\eeqn
$\partial^*E$ denoting the reduced boundary of $E$ and $\nu_E$ the measure-theoretic exterior unit normal to $E$ (see for instance \cite{Maggi2012}). We recall that volume-constrained minimizers of the surface energy $\per_\psi$ are obtained by translations and dilations of the \emph{Wulff shape} of $\psi$, that is, the open, bounded, convex set
\beqn \label{eq:wulff}
W_\psi \defeq \bigcap_{\nu\in\s^{d-1}} \Bigl\{ x\in\R^d \,:\, x\cdot\nu < \psi(\nu) \Bigr\}.
\eeqn
Conversely, any open, bounded, convex set $K\subset\R^d$ which contains the origin is the Wulff shape of a surface tension $\psi$ with the properties above: one has that $K=W_\psi$ for
\beqn \label{eq:psi}
\psi(\nu)=\sup\bigl\{ x\cdot\nu : x\in K \bigr\}.
\eeqn
We also introduce the dual $\psi_*:\R^d\to[0,\infty)$ of $\psi$ defined by
\begin{equation} \label{eq:dual}
\psi_*(\xi) \defeq \sup\bigl\{ \xi\cdot x \,:\, \psi(x)<1 \bigr\}
\end{equation}
and we remark that the Wulff shape of $\psi$ coincides with the unit ball of $\psi_*$, that is,
\begin{equation} \label{eq:dualW}
W_\psi=\bigl\{\xi\in\R^d \,:\, \psi_*(\xi)<1 \bigr\}.
\end{equation}
We finally recall the \emph{quantitative Wulff inequality}, which states that for every set of finite perimeter $E\subset\R^d$ with $|E|=|W_\psi|$ one has the sharp stability estimate
\beqn \label{eq:quantisop}
\per_\psi (E) \geq \per_\psi(W_\psi) + \bar{c} \min_{x_0\in\R^d} |\asymm{E}{(x_0+W_\psi)}|^2,
\eeqn
where $\asymm{E}{F}\defeq (E\setminus F)\cup(F\setminus E)$ denotes the symmetric difference of sets, and $\bar{c}>0$ is a dimensional constant (depending also on the volume of $W_\psi$ since the inequality is not displayed in a scaling invariant form). The sharp inequality \eqref{eq:quantisop} was proved by Figalli, Maggi and Pratelli in \cite{FigalliMaggiPratelliINVENTIONES} and further extended in a stronger form by Neumayer in \cite{Neum16} (see also \cite{CicaleseLeonardi,fuscomaggipratelli} for the quantitative isoperimetric inequality in the isotropic case).

Here we will consider only the class of \emph{crystalline surface tensions}, i.e.\ when there exists a finite set $\{\xi_1,\ldots,\xi_N\}\subset\R^d\setminus\{0\}$, $N\in\N$, such that
\beqn \label{eq:psicryst}
\psi(\nu) = \max_{1\leq i\leq N} \nu\cdot\xi_i \qquad\text{for all }\nu\in\R^d.
\eeqn
Notice that the corresponding Wulff shape is a convex polyhedron.

\medskip
The second term in the energy functional $\en$ is the nonlocal repulsive interaction 
\beqn \label{eq:nonlocal}
\nl(E) \defeq \int_E\int_E \frac{1}{|x-y|^\alpha}\dd x \dd y \,,
\eeqn
where the parameter $\alpha$ ranges in the interval $(0,d)$. It will be convenient to define the interaction between two measurable sets $E,F\subset\R^d$ as
\beqn\ \label{eq:nonlocali}
\nli(E,F) \defeq \int_E\int_F \frac{1}{|x-y|^\alpha}\dd x \dd y \,,
\eeqn
so that $\nl(E)=\nli(E,E)$.
We also introduce, for a Borel set $E\subset\R^d$, the associated \emph{Riesz potential} defined by
\begin{equation} \label{eq:potential}
v_E(x) \defeq \int_E \frac{1}{|x-y|^\alpha}\dd y \qquad\qquad\text{for }x\in\R^d.
\end{equation}
Notice in particular that $\nl(E)=\int_E v_E(x)\dd x$.

\medskip
In this paper we will be mainly interested in the minimization problem \eqref{eq:min} for $\gamma\ll1$. This corresponds to the small mass regime, when minimizing with respect to a mass constraint $|E|=m\in(0,\infty)$: indeed for $E$ with $|E|=m$, setting $\widetilde{E}=m^{-1/d}E$, by a scaling argument one finds
\begin{equation*}
\en(E) = m^{\frac{d-1}{d}}\Bigl[ \per_\psi(\widetilde{E}) + \gamma m^{\frac{d+1-\alpha}{d}}\nl(\widetilde{E})\Bigr].
\end{equation*}
This implies that minimizing $\en$ under the volume constraint $|E|=m$ is equivalent to solving \eqref{eq:min} with $\tilde{\gamma}=\gamma m^{\frac{d+1-\alpha}{d}}$.

\subsection{A class of symmetric polygons, and their variations}
In the following we fix the space dimension $d=2$. Most of the arguments can be generalized to higher space dimensions, but for clarity of the exposition we restrict to the planar case and we postpone the discussion of the possible extension to $d>2$ to Section~\ref{sec:3d}.

Given an open, convex polygon $P\subset\R^2$ with $n$ sides, we will denote in the following by $L_1,\ldots,L_n$ the sides of $P$, and by $\ell_i$ the length of $L_i$. By translation invariance we will always assume without loss of generality that $P$ contains the origin, so that we can consider the corresponding crystalline surface density $\psi$ (according to \eqref{eq:psi}), with $P=W_\psi$. Notice that $\psi$ can be represented as in \eqref{eq:psicryst} for suitable vectors $\xi_i$. We now introduce the class of polygons for which our main result holds.

\begin{definition} \label{def:polygons}
Let $\pol$, $n\geq3$, be the class of open, convex polygons $P\subset\R^2$ with $n$ sides $L_1,\ldots,L_n$ and unit area $|P|=1$, which are reflection symmetric with respect to the bisectors of all angles.
\end{definition}

\begin{remark} \label{rmk:polygons}
For every $n\geq3$, $\pol$ contains the regular polygon with $n$ sides. Each polygon in the class $\pol$ is equilateral, that is $\ell_i=\ell_j$ for every $i,j\in\{1,\ldots,n\}$. Moreover, the internal angles of any polygon in $\pol$ can only take two alternating values $\alpha,\beta\in(0,\pi)$ (possibly equal). In particular, it follows that if $n$ is odd, then $\pol$ contains only the regular polygon with $n$ sides. If $n$ is even, then $\pol$ contains a one-parameter family of polygons $P_{\alpha}$, which can be parametrized by one of the interior angles $\alpha\in(\frac{(n-4)\pi}{n},\pi)$; the other angle is determined by the constraint $\frac{n}{2}(\alpha+\beta)=\pi(n-2)$ (the requirement $\alpha>\frac{(n-4)\pi}{n}$ follows by the condition $\beta<\pi$). The polygon $P_\alpha$ can be constructed as follows: one considers a triangle with two angles equal to $\frac{\alpha}{2}$ and $\frac{\beta}{2}$ respectively, and places $n$ copies of such triangle next to each other so that the equal angles are adjacent.
\end{remark}

\begin{remark} \label{rmk:isometry}
Notice that each polygon in the class $\pol$ has the property that, for every pair of sides $L_i$, $L_j$, $i,j\in\{1,\ldots,n\}$, there exists an isometry $S_{ij}:\R^2\to\R^2$ which maps $L_i$ onto $L_j$ and leaves the polygon invariant, i.e.\ $S_{ij}(L_i)=L_j$ and $S_{ij}(P)=P$; such isometry can be obtained as a composition of reflections with respect to the bisectors of the angles.
The existence of the isometries $S_{ij}$ characterizes the polygons in the class $\pol$ and is the property on which our main result (also in higher dimension) strongly relies.
\end{remark}

\begin{figure}
\definecolor{qqttcc}{rgb}{0,0.2,0.8}
\begin{tikzpicture}[scale=0.15,line cap=round,line join=round,>=triangle 45,x=1cm,y=1cm]
\clip(-60,-1.5) rectangle (5,20);
\fill[line width=2pt,color=qqttcc,fill=qqttcc,fill opacity=0.16] (-50,15.019893331027108) -- (-54.20581638041034,7.509946665513554) -- (-50,0) -- (-45.79418361958969,7.509946665513554) -- cycle;
\fill[line width=2pt,color=qqttcc,fill=qqttcc,fill opacity=0.16] (-0.4319962961640007,0.7276888756272565) -- (-7.2125942244159065,0.15514799233725685) -- (-13.989600315240022,0.7687437652899137) -- (-14.562141198530021,7.549341693541819) -- (-13.948545425577365,14.326347784365932) -- (-7.1679474973254615,14.898888667655932) -- (-0.3909414065013497,14.285292894703275) -- (0.18159947678864974,7.504694966451374) -- cycle;
\fill[line width=2pt,color=qqttcc,fill=qqttcc,fill opacity=0.16] (-30,0) -- (-37.35241332883696,2.6158234616359755) -- (-36.05306327931321,10.3107700884364) -- (-30.111487045302624,15.370235079514527) -- (-24.097142810489963,10.397491614818177) -- (-22.68630571566359,2.7222031621040754) -- cycle;
\draw [line width=1pt,color=qqttcc] (-50,15.019893331027108)-- (-54.20581638041034,7.509946665513554);
\draw [line width=1pt,color=qqttcc] (-54.20581638041034,7.509946665513554)-- (-50,0);
\draw [line width=1pt,color=qqttcc] (-50,0)-- (-45.79418361958969,7.509946665513554);
\draw [line width=1pt,color=qqttcc] (-45.79418361958969,7.509946665513554)-- (-50,15.019893331027108);
\draw [line width=1pt,color=qqttcc] (-0.4319962961640007,0.7276888756272565)-- (-7.2125942244159065,0.15514799233725685);
\draw [line width=1pt,color=qqttcc] (-7.2125942244159065,0.15514799233725685)-- (-13.989600315240022,0.7687437652899137);
\draw [line width=1pt,color=qqttcc] (-13.989600315240022,0.7687437652899137)-- (-14.562141198530021,7.549341693541819);
\draw [line width=1pt,color=qqttcc] (-14.562141198530021,7.549341693541819)-- (-13.948545425577365,14.326347784365932);
\draw [line width=1pt,color=qqttcc] (-13.948545425577365,14.326347784365932)-- (-7.1679474973254615,14.898888667655932);
\draw [line width=1pt,color=qqttcc] (-7.1679474973254615,14.898888667655932)-- (-0.3909414065013497,14.285292894703275);
\draw [line width=1pt,color=qqttcc] (-0.3909414065013497,14.285292894703275)-- (0.18159947678864974,7.504694966451374);
\draw [line width=1pt,color=qqttcc] (0.18159947678864974,7.504694966451374)-- (-0.4319962961640007,0.7276888756272565);
\draw [line width=1pt,color=qqttcc] (-30,0)-- (-37.35241332883696,2.6158234616359755);
\draw [line width=1pt,color=qqttcc] (-37.35241332883696,2.6158234616359755)-- (-36.05306327931321,10.3107700884364);
\draw [line width=1pt,color=qqttcc] (-36.05306327931321,10.3107700884364)-- (-30.111487045302624,15.370235079514527);
\draw [line width=1pt,color=qqttcc] (-30.111487045302624,15.370235079514527)-- (-24.097142810489963,10.397491614818177);
\draw [line width=1pt,color=qqttcc] (-24.097142810489963,10.397491614818177)-- (-22.68630571566359,2.7222031621040754);
\draw [line width=1pt,color=qqttcc] (-22.68630571566359,2.7222031621040754)-- (-30,0);
\end{tikzpicture}
\caption{Examples of (non regular) polygons in the class $\pol$.}
\end{figure}
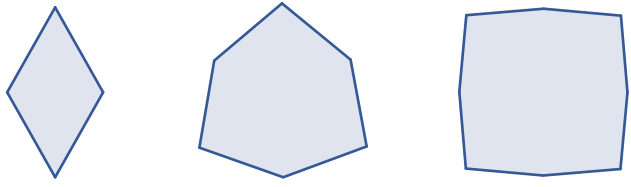

A key tool for our analysis is a result by Figalli and Maggi \cite[Theorem~7]{FigalliMaggiARMA} (recently extended by Figalli and Zhang \cite{FigZha} to any space dimension), which states that for a crystalline surface tension $\psi$, every $(\omega,R)$-minimizer of the anisotropic perimeter $\per_\psi$ is a convex polygon with sides parallel to those of the Wulff shape of $\psi$, provided $\omega$ is sufficiently small and $R\geq d+1$. Given $\omega,R>0$, a (volume constrained) $(\omega,R)$-minimizer of $\per_\psi$ is a set of finite perimeter $E$ satisfying the inequality
\beqn \label{eq:quasimin}
\per_\psi(E) \leq \per_\psi(F) + \omega |\asymm{E}{F}|
\eeqn
for all sets of finite perimeter $F\subset\R^2$ such that $|F|=|E|$ and $F\subset I_R(E)$, where $I_R(E)$ is the $R$-neighbourhood of $E$ with respect to $\psi_*$, i.e.
\begin{equation} \label{eq:distance}
I_R(E) \defeq \bigl\{ x\in\R^2\,:\, \mathrm{dist}_{W_\psi}(x,E)<R \bigr\}, \qquad\text{where}\quad\mathrm{dist}_{W_\psi}(x,E) = \inf_{y\in E}\psi_*(x-y).
\end{equation}
We also say that $E$ is an $\omega$-minimizer if it is an $(\omega,R)$-minimizer with $R=\infty$.

Since it turns out that minimizers of \eqref{eq:min} are $\omega$-minimizers of $\per_\psi$ with a constant $\omega$ proportional to $\gamma$ (see the proof of Theorem~\ref{thm:main1}), in view of the Figalli-Maggi result for $\gamma\ll1$ we are allowed to restrict to a finite-dimensional class of competitors, i.e. the class of all open, convex polygons which are close to $W_\psi$ and whose sides are parallel to those of $W_\psi$. We now introduce some notation to deal with the polygons in this class.

If $\psi$ is crystalline ($W_\psi$ is a convex polygon), then the dual $\psi_*$ (see \eqref{eq:dual}) is of the form
\begin{equation} \label{eq:dualcryst}
\psi_*(\xi) = \sup_{1\leq i\leq n} \xi\cdot\sigma_i\,,
\end{equation}
for some $n\in\N$ and vectors $\sigma_i\in\R^2$. We assume here that the set $\{\sigma_i\}_{i=1}^n$ is minimal, that is, denoting by $V_i$ the convex cone
\begin{equation*}
V_i \defeq \bigl\{ \xi\in\R^2 \,:\, \psi_*(\xi)=\xi\cdot\sigma_i \bigr\},
\end{equation*}
one has that $|W_\psi\cap V_i|>0$ for all $i$. The vector $\sigma_i$ is parallel to the exterior normal $\nu_i$ of the side $\partial W_\psi\cap V_i$ of $W_\psi$.
Following \cite{FigZha}, we can define the class of competitors by changing the length of the vectors $\sigma_i$: geometrically, this reflects in a translation of the sides of $W_\psi$ which keeps the orientation of the normals.

\begin{definition} \label{def:fzclass}
Given a polygon $P\in\pol$, let $\psi$ the corresponding surface tension (such that $W_\psi=P$) and $\psi_*$ its dual, given by \eqref{eq:dualcryst}. For $\mathbf{d}=(d_1,\ldots,d_n)\in\R^n$ with $|d_i|<\frac{1}{|\sigma_i|}$ let
\begin{equation} \label{eq:fzpsi}
\psi_*^{\mathbf{d}}(\xi) \defeq \sup_{1\leq i\leq n} \xi\cdot\frac{\sigma_i}{1+d_i|\sigma_i|},
\qquad
P(\mathbf{d}) \defeq \bigl\{\xi\in\R^2 \,:\, \psi_*^{\mathbf{d}}(\xi)<1 \bigr\}.
\end{equation}
Then for $\e>0$ we define the class of competitors
\begin{equation} \label{eq:fz}
\FZ(P,\e) \defeq \Bigl\{ P(\mathbf{d}) \,:\, \mathbf{d}\in\R^n, \, |\mathbf{d}|_\infty<\e, \, |P(\mathbf{d})|=|P| \Bigr\} \,,
\end{equation}
where $|\mathbf{d}|_\infty\defeq\sup_{1\leq i\leq n}|d_i|$.
\end{definition}

Notice that $P(\mathbf{0})=P$. Notice also that $P(\mathbf{d})$ is obtained by translating the side $L_i=\partial P\cap V_i$ by the vector $d_i\nu_i$. The geometric meaning of the parameters $\mathbf{d}=(d_1,\ldots,d_n)$ is explained in Figure~\ref{fig:d}.

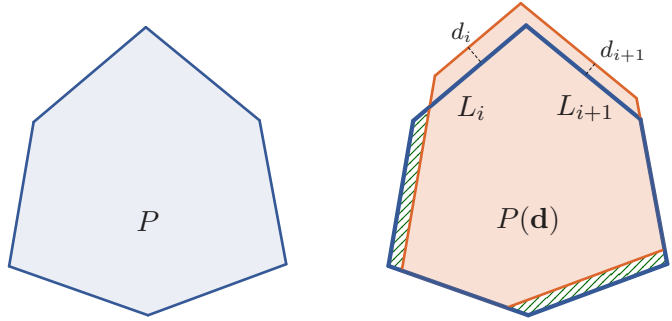
\begin{figure}
\definecolor{qqwuqq}{rgb}{0,0.4,0}
\definecolor{yqqqqq}{rgb}{0.9,0.2,0}
\definecolor{qqttcc}{rgb}{0,0.2,0.8}
\begin{tikzpicture}[scale=0.25,line cap=round,line join=round,>=triangle 45,x=1cm,y=1cm]
\clip(-60,-1) rectangle (-21,20);
\fill[line width=1.2pt,color=qqttcc,fill=qqttcc,fill opacity=0.1] (-50,0) -- (-57.30962693713489,2.5979273782193033) -- (-56.02047131036041,10.247630012475428) -- (-50.11578673506794,15.278988943111623) -- (-44.135527735792635,10.33769610376521) -- (-42.73058537395022,2.708409794909705) -- cycle;
\fill[line width=2pt,color=yqqqqq,fill=yqqqqq,fill opacity=0.2] (-31.131523932027932,0.4025707908983199) -- (-36.65185428412461,2.3665802943302943) -- (-34.90689834116111,12.700472113662478) -- (-30.389946189260723,16.546818668655494) -- (-24.302272414960626,11.513445320796713) -- (-22.828024277964225,3.49318575729532) -- cycle;
\fill[line width=0pt,color=qqwuqq,fill=qqwuqq,pattern=north east lines,pattern color=qqwuqq] (-36.65185428412461,2.3665802943302943) -- (-37.35241332883696,2.6158234616359755) -- (-36.05306327931321,10.3107700884364) -- (-35.18570112868179,11.049360028692657) -- cycle;
\fill[line width=0pt,color=qqwuqq,fill=qqwuqq,pattern=north east lines,pattern color=qqwuqq] (-31.131523932027932,0.4025707908983199) -- (-30,0) -- (-22.68630571566359,2.7222031621040754) -- (-22.828024277964225,3.49318575729532) -- cycle;
\draw [line width=1pt,color=qqttcc] (-50,0)-- (-57.30962693713489,2.5979273782193033);
\draw [line width=1pt,color=qqttcc] (-57.30962693713489,2.5979273782193033)-- (-56.02047131036041,10.247630012475428);
\draw [line width=1pt,color=qqttcc] (-56.02047131036041,10.247630012475428)-- (-50.11578673506794,15.278988943111623);
\draw [line width=1pt,color=qqttcc] (-50.11578673506794,15.278988943111623)-- (-44.135527735792635,10.33769610376521);
\draw [line width=1pt,color=qqttcc] (-44.135527735792635,10.33769610376521)-- (-42.73058537395022,2.708409794909705);
\draw [line width=1pt,color=qqttcc] (-42.73058537395022,2.708409794909705)-- (-50,0);
\draw [line width=1.6pt,color=qqttcc] (-30,0)-- (-37.35241332883696,2.6158234616359755);
\draw [line width=1.6pt,color=qqttcc] (-37.35241332883696,2.6158234616359755)-- (-36.05306327931321,10.3107700884364);
\draw [line width=1.6pt,color=qqttcc] (-36.05306327931321,10.3107700884364)-- (-30.111487045302624,15.370235079514527);
\draw [line width=1.6pt,color=qqttcc] (-30.111487045302624,15.370235079514527)-- (-24.097142810489963,10.397491614818177);
\draw [line width=1.6pt,color=qqttcc] (-24.097142810489963,10.397491614818177)-- (-22.68630571566359,2.7222031621040754);
\draw [line width=1.6pt,color=qqttcc] (-22.68630571566359,2.7222031621040754)-- (-30,0);
\draw [line width=1pt,color=yqqqqq] (-31.131523932027932,0.4025707908983199)-- (-36.65185428412461,2.3665802943302943);
\draw [line width=1pt,color=yqqqqq] (-36.65185428412461,2.3665802943302943)-- (-34.90689834116111,12.700472113662478);
\draw [line width=1pt,color=yqqqqq] (-34.90689834116111,12.700472113662478)-- (-30.389946189260723,16.546818668655494);
\draw [line width=1pt,color=yqqqqq] (-30.389946189260723,16.546818668655494)-- (-24.302272414960626,11.513445320796713);
\draw [line width=1pt,color=yqqqqq] (-24.302272414960626,11.513445320796713)-- (-22.828024277964225,3.49318575729532);
\draw [line width=1pt,color=yqqqqq] (-22.828024277964225,3.49318575729532)-- (-31.131523932027932,0.4025707908983199);
\draw [line width=1pt,color=qqttcc] (-30,0)-- (-37.35241332883696,2.6158234616359755);
\draw [line width=1pt,color=qqttcc] (-37.35241332883696,2.6158234616359755)-- (-36.05306327931321,10.3107700884364);
\draw [line width=1pt,color=qqttcc] (-36.05306327931321,10.3107700884364)-- (-30.111487045302624,15.370235079514527);
\draw [line width=1pt,color=qqttcc] (-30.111487045302624,15.370235079514527)-- (-24.097142810489963,10.397491614818177);
\draw [line width=1pt,color=qqttcc] (-24.097142810489963,10.397491614818177)-- (-22.68630571566359,2.7222031621040754);
\draw [line width=1pt,color=qqttcc] (-22.68630571566359,2.7222031621040754)-- (-30,0);
\draw [line width=0.4pt,dash pattern=on 1pt off 1pt] (-33.1497558180976,14.196741913061075)-- (-32.45193642268202,13.377258599559706);
\draw [line width=0.4pt,dash pattern=on 1pt off 1pt] (-26.502913479735923,13.332965961078893)-- (-26.967658799326706,12.770874159364029);
\node at (-50,5) {$P$};
\node at (-30,5) {$P(\mathbf{d})$};
\node at (-33,11) {$L_i$};
\node at (-27,11) {$L_{i+1}$};
\node at (-33.5,15) {{\footnotesize $d_i$}};
\node at (-25,14) {{\footnotesize $d_{i+1}$}};
\end{tikzpicture}
\caption{A polygon $P\in\pol$ (left) and a variation $P(\mathbf{d})$ in the class $\FZ(P,\e)$ (right).}\label{fig:d}
\end{figure}

\subsection{Main results}
In our first main result we show that each polygon in the class $\pol$ is a global minimizer of the energy $\en$ (with respect to its own surface tension), provided that $\gamma$ is sufficiently small.

\begin{theorem}[Minimality of polygons in $\pol$] \label{thm:main1}
Let $P\in\pol$ and let $\psi$ be a surface energy density whose Wulff shape is $P$. Then there exists $\bar\gamma>0$, depending on $P$ and $\alpha$, such that for all $\gamma<\bar{\gamma}$ the polygon $P$ is the unique (up to translations) solution to \eqref{eq:min}.
\end{theorem}

The proof of Theorem~\ref{thm:main1}, given in Section~\ref{sec:proofs}, follows by the combination of three main ingredients: (a) the stability of the Wulff inequality \eqref{eq:quantisop}; (b) the fact that any solution to \eqref{eq:min} is an $\omega$-minimizer of the anisotropic perimeter and in turn, if $\gamma$ is sufficiently small, it is a polygon with sides parallel to those of $P$ (that is, it belongs to the class $\FZ(P,\e)$); (c) the following quadratic upper bound for variations within the class $\FZ(P,\e)$, which is one of the main new contributions of this paper and will be proved in Section~\ref{sec:quadratic}.

\begin{mainlemma}[Quadratic bound]
Let $P\in\pol$. There exists $\e_0>0$ and $c_0>0$ (depending on the polygon $P$ and on $\alpha$) such that for every $\widetilde{P}\in\FZ(P,\e_0)$ one has the quadratic estimate
\beqn \label{eq:quadratic}
\big| \nl(P) - \nl(\widetilde{P}) \big| \leq c_0 |\asymm{P}{\widetilde{P}}|^2 \,.
\eeqn
\end{mainlemma}

It is clear that the polygons in $\pol$ cease to be global minimizers of $\en$ for $\gamma$ large enough; at least in the case of the square, it is known that the Wulff shape loses its stability for large values of $\gamma$ and hence it is not even a local minimizer (see \cite{ChNeuTo20}).

\begin{remark}[Criticality conditions]\label{rmk:critical}
Let $P$ be a convex polygon (not necessarily in the class $\pol$) with $n$ sides $\{L_i\}_{i=1}^n$ and corresponding side lengths $\{\ell_i\}_{i=1}^n$. By considering the first variation of the nonlocal energy $\nl(\cdot)$ with respect to perturbations in the class $\FZ(P,\e)$, we say that $P$ is \emph{critical} for $\nl(\cdot)$ with respect to this class of variations if
\begin{equation} \label{eq:critical}
\frac{1}{\ell_i}\int_{L_i} v_{P}(x)\dd\Hone(x) = \frac{1}{\ell_j}\int_{L_j}v_P(x)\dd\Hone(x) \qquad\text{for every }i,j=1,\ldots,n,
\end{equation}
where $v_P$ is the potential associated to $P$ according to \eqref{eq:potential}. A derivation of these equations will be given in Section~\ref{sec:proofs}.

Notice that $P$ is always critical for the anisotropic perimeter $\per_\psi$ (where $\psi$ is such that $W_\psi=P$), which is differentiable with respect to this family of perturbations as they do not modify the orientation of the normals to the boundary. Hence the equations \eqref{eq:critical} are necessary conditions for minimality of $P$ for $\en$ in the class $\FZ(P,\e)$ (and for global minimality when $\gamma$ is small).
While it is obvious that each polygon $P\in\pol$ is critical for $\nl(\cdot)$ with respect to this class of perturbations, it is an open question whether there exist polygons $P\notin\pol$ for which the conditions \eqref{eq:critical} hold.
\end{remark}

In the next theorem we observe that, in the planar case, the rigid structure of minimizers of $\en$ is not just peculiar to the small mass regime ($\gamma\ll1$), but it characterizes (local) minimizers also for large $\gamma$. This property was first observed in \cite[Theorem~14]{FigalliMaggiARMA} for a similar minimization problem, where the crystalline perimeter was perturbed by a bulk potential energy induced by an external force field. Their argument applies also to the energy $\en$ with minor modifications, see Section~\ref{sec:proofs} for the proof.

\begin{theorem}[Rigidity of minimizers] \label{thm:main2}
Let $E\subset\R^2$ be a \emph{local minimizer} of $\en$, in the sense that there exists $\delta_0>0$ such that $\en(E)\leq\en(F)$ for every set of finite perimeter $F\subset\R^2$ such that $|F|=|E|$ and $\min_{y\in\R^2}|\asymm{(y+E)}{F}|<\delta_0$.

Then there exists a constant $v_0\in\R$ such that $\partial E = \Gamma_1\cup\Gamma_2$, where
\begin{equation*}
\Gamma_1=\bigl\{ x\in\R^2\,:\, v_E(x)=v_0 \bigr\}, \qquad \nu_E(x)\in\{\nu_i\}_{i=1}^n \quad\text{for $\Hone$-a.e. }x\in\Gamma_2,
\end{equation*}
where $v_E$ is the Riesz potential associated to $E$ (according to \eqref{eq:potential}), and $\{\nu_i\}_{i=1}^n$ is the set of exterior normals to the boundary of the Wulff-shape $W_\psi$.
\end{theorem}

Obviously, the theorem applies to global minimizers of $\en$, whenever they exist. Notice that the existence of solutions to the minimum problem \eqref{eq:min} was established in \cite[Theorem~3.1]{ChNeuTo20} for small values of $\gamma$. However, it is well-known that, at least for large values of $\gamma$ (corresponding to the large mass regime, see \cite[Theorem~3.1]{ChNeuTo20}), global minimizers do not exist due to the possibility of splitting the mass into different connected components and moving them far apart from each other: this reduces the cross interactions between the different components, and a ``minimizing configuration'' is reached only when the components are at infinite distance from each other. Following \cite{KnMuNo2016}, this situation can be captured by introducing a notion of \emph{generalized minimizer}: a collection of sets of finite perimeter $(E_1,\ldots,E_N)$, $N\in\N$, such that 
\begin{equation} \label{eq:genmin}
\sum_{i=1}^N |E_i|=1, \qquad \sum_{i=1}^N\en(E_i) = \inf\bigl\{ \en(E) \,:\, |E|=1 \bigr\}.
\end{equation}
The existence of a generalized minimizer can be proved by arguing as in \cite[Theorem~4.5]{KnMuNo2016}. Since by \eqref{eq:genmin} each component $E_i$ is a solution of the minimum problem corresponding to its mass, i.e.
$$
\en(E_i)=\min\bigl\{ \en(E) \,:\, |E|=|E_i| \bigr\},
$$
the result in Theorem~\ref{thm:main2} also applies to each component of a generalized minimizer.


\section{Proof of the quadratic bound}\label{sec:quadratic}

We first gather some well-known properties of the nonlocal energy \eqref{eq:nonlocal}. 
As a consequence of the classical Riesz rearrangement inequality (see for instance \cite{LiLo}), one can show that for every Borel set $E\subset\R^d$ with finite Lebesgue measure and every point $x\in\R^d$ it holds
\begin{equation} \label{eq:potential2}
v_E(x) \leq \int_{B_r}\frac{1}{|y|^\alpha}\dd y \,, \qquad\text{where } r=\biggl(\frac{|E|}{|B_1|}\biggr)^\frac{1}{d},
\end{equation}
see \cite[Lemma~2.3]{FuPr} for a proof. The following lemma is an immediate consequence of the bound \eqref{eq:potential2}.

\begin{lemma} \label{lem:potential1}
There exists a constant $c_{\alpha,d}$, depending only on $\alpha$ and on the space dimension, such that for any two Borel sets $E,F\subset\R^d$ with finite Lebesgue measure one has
\begin{equation} \label{eq:nlbound1}
\| v_E \|_\infty \leq c_{\alpha,d} |E|^{1-\frac{\alpha}{d}},
\end{equation}
\begin{equation} \label{eq:nlbound2}
\nli(E,F) \leq c_{\alpha,d} |E|^{1-\frac{\alpha}{d}} |F| \,,
\end{equation}
\beqn \label{eq:lipschitz}
\big| \nl(E)-\nl(F)\big| \leq c_{\alpha,d} \Bigl( |E|^{1-\frac{\alpha}{d}} + |F|^{1-\frac{\alpha}{d}} \Bigr) |\asymm{E}{F}| \,.
\eeqn
\end{lemma}

\begin{proof}
By setting $r=(|E|/|B_1|)^{1/d}$ we have thanks to \eqref{eq:potential2}
\begin{align*}
v_E(x) \leq \int_{B_r}\frac{1}{|y|^\alpha}\dd y = d\,|B_1|\int_0^r \rho^{d-1-\alpha}\dd\rho = \frac{d\,|B_1|}{d-\alpha} r^{d-\alpha} = \frac{d\,|B_1|^{\frac{\alpha}{d}}}{d-\alpha}|E|^{1-\frac{\alpha}{d}}
\end{align*}
for all $x\in\R^d$, which proves \eqref{eq:nlbound1}. Recalling that $\nli(E,F)=\int_F v_E(x)\dd x$, we obtain \eqref{eq:nlbound2} by integration over $F$.
Finally, we use once more \eqref{eq:nlbound1} to obtain
\begin{align*}
\big| \nl(E)-\nl(F)\big|
& = \bigg| \int_{\R^d}\int_{\R^d}\biggl( \frac{\chi_E(x)\bigl(\chi_E(y)-\chi_F(y)\bigr)}{|x-y|^\alpha} + \frac{\chi_F(y)\bigl(\chi_E(x)-\chi_F(x)\bigr)}{|x-y|^\alpha} \biggr)\dd x \dd y  \bigg| \\
& = \bigg| \int_{E\setminus F} \bigl( v_E(x)+v_F(x)\bigr)\dd x - \int_{F\setminus E} \bigl( v_E(x)+v_F(x)\bigr)\dd x\bigg| \\
& \leq c_{\alpha,d} \Bigl( |E|^{1-\frac{\alpha}{d}} + |F|^{1-\frac{\alpha}{d}} \Bigr) |\asymm{E}{F}| \,,
\end{align*}
which is \eqref{eq:lipschitz}.
\end{proof}

Let $P\in\pol$ be a fixed polygon and let $\psi_*$ the dual of the corresponding anisotropy $\psi$, given by the formula \eqref{eq:dualcryst}. As already remarked, each vector $\sigma_i$ is parallel to the normal $\nu_i$ of the side $L_i=\partial P\cap V_i$ of $P$, for $i=1,\ldots,n$. We number the sides of $P$ in clockwise order.

With the notation introduced in Section~\ref{sec:main}, we consider any possible variation $P(\mathbf{d})$ in the class $\FZ(P,\e)$, where $\mathbf{d}=(d_1,\ldots,d_n)$. Notice that the area constraint $|P(\mathbf{d})|=|P|$ reduces by one the number of free variables: we consider the last variable $d_n$ as a function of the first $(n-1)$ variables,
\begin{equation} \label{eq:volumeadj}
d_n = f(d_1,\ldots,d_{n-1}),
\end{equation}
chosen in order to restore the area constraint, which is possible if $\e$ is small enough.
It is convenient to introduce an auxiliary function $V:\{(d_1,\ldots,d_{n-1})\in\R^{n-1}:|d_i|<\e\}\to\R$,
\begin{equation} \label{eq:V}
V(d_1,\ldots,d_{n-1}) \defeq \nl \bigl( P(d_1,\ldots,d_{n-1},f(d_1,\ldots,d_{n-1}) ) \bigr).
\end{equation}

The proof of the Main Lemma will be achieved by an iteration argument, which allows us to prove the estimate \eqref{eq:quadratic} by moving just one side at a time, keeping the others fixed. 

\begin{lemma} \label{lem:induction0}
There exist $\e_1>0$ and $c_1>0$ such that for every $|d_1|<\e_1$ one has
\begin{equation} \label{eq:induction0}
\big| V(d_1,0,\ldots,0) - V(0,\ldots,0) \big| \leq c_1 |d_1|^2 .
\end{equation}
\end{lemma}

\begin{proof}
Along this proof we will denote by $O(d_1^\beta)$, for $\beta>0$, any quantity with the property
\begin{equation*}
O(d_1^\beta) \leq C|d_1|^\beta \qquad\text{for all }|d_1|<\e_1
\end{equation*}
(for $\e_1\in(0,1)$ to be chosen small enough), for a constant $C>0$ depending only on $P$.

Fix $|d_1|<\e_1$ and let $d_n=f(d_1,0,\ldots,0)$ be chosen such that $|\widetilde{P}|=|P|$, where we set $\widetilde{P}=P(d_1,0,\ldots,0,d_n)$. Geometrically, we are translating the two adjacent sides $L_1$ and $L_n$ parallel to themselves by $d_1$ and $d_n$ respectively (towards the exterior of the polygon if $d_i$ is positive and towards the interior otherwise). We consider the case $d_1>0$ (and therefore $d_n<0$) in detail, see Figure~\ref{fig:d1}; the case $d_1<0$ will follow by a symmetric argument.
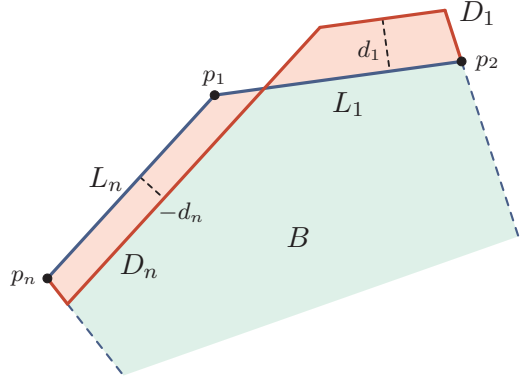
\begin{figure}
\definecolor{yqqqqq}{rgb}{0.8,0,0}
\definecolor{qqttzz}{rgb}{0,0.2,0.6}
\definecolor{qqccqq}{rgb}{0,0.8,0}
\definecolor{zzttqq}{rgb}{1,0.2,0}
\begin{tikzpicture}[scale=1.1,line cap=round,line join=round,>=triangle 45,x=1cm,y=1cm]
\clip(-2.8,-1.2) rectangle (4,3.5);
\fill[line width=0pt,color=zzttqq,fill=zzttqq,fill opacity=0.19] (-2,0) -- (-1.759911892164874,-0.30903305671090253) -- (0.5939518414006387,2.2908966181133485) -- (0,2.209074074892186) -- cycle;
\fill[line width=0pt,color=zzttqq,fill=zzttqq,fill opacity=0.2] (0.5939518414006387,2.2908966181133485) -- (1.2595344211689785,3.0260572289463985) -- (2.7537130477244114,3.2318946163062305) -- (2.9520543304584077,2.615747775028463) -- cycle;
\fill[line width=0pt,color=qqccqq,fill=qqccqq,fill opacity=0.14] (-1.0901880239022148,-1.1710783117952912) -- (-1.759911892164874,-0.30903305671090253) -- (0.5939518414006387,2.2908966181133485) -- (2.9520543304584077,2.615747775028463) -- (3.6309973338521573,0.5066125451744923) -- cycle;
\draw [line width=0.8pt,color=qqttzz,dash pattern=on 3pt off 3pt] (2.9520543304584077,2.615747775028463)-- (3.6309973338521573,0.5066125451744923);
\draw [line width=0.8pt,color=qqttzz,dash pattern=on 3pt off 3pt] (-1.759911892164874,-0.30903305671090253)-- (-1.0901880239022148,-1.1710783117952912);
\draw [line width=0.7pt,dash pattern=on 2pt off 2pt] (-0.8929113890626539,1.2228203745150465)-- (-0.607215058335215,0.9641632892008002);
\draw [line width=0.7pt,dash pattern=on 2pt off 2pt] (2.0066237344466953,3.128975922626314)-- (2.093616980223127,2.49748982497876);
\draw [line width=1.2pt,color=qqttzz] (-2,0)-- (0,2.209074074892186);
\draw [line width=1.2pt,color=qqttzz] (0,2.209074074892186)-- (2.9520543304584077,2.615747775028463);
\draw [line width=1.2pt,color=yqqqqq] (-2,0)-- (-1.759911892164874,-0.30903305671090253);
\draw [line width=1.2pt,color=yqqqqq] (-1.759911892164874,-0.30903305671090253)-- (1.2595344211689785,3.0260572289463985);
\draw [line width=1.2pt,color=yqqqqq] (1.2595344211689785,3.0260572289463985)-- (2.7537130477244114,3.2318946163062305);
\draw [line width=1.2pt,color=yqqqqq] (2.7537130477244114,3.2318946163062305)-- (2.9520543304584077,2.615747775028463);
\begin{scriptsize}
\draw [fill=black] (0,2.209074074892186) circle (1.5pt);
\draw [fill=black] (2.9520543304584077,2.615747775028463) circle (1.5pt);
\draw [fill=black] (-2,0) circle (1.5pt);
\end{scriptsize}
\node at (1.6,2.1) {$L_1$};
\node at (-1.3,1.2) {$L_n$};
\node [left] at (-2,0) {{\footnotesize $p_n$}};
\node [above] at (0,2.2) {{\footnotesize $p_1$}};
\node [right] at (3,2.6) {{\footnotesize $p_2$}};
\node at (1,0.5) {$B$};
\node at (1.85,2.75) {{\footnotesize $d_1$}};
\node at (-0.4,0.8) {{\footnotesize $-d_n$}};
\node at (-0.9,0.15) {$D_n$};
\node [right] at (2.8,3.2) {$D_1$};
\end{tikzpicture}
\caption{The variation considered in Lemma~\ref{lem:induction0} (case $d_1>0$).} \label{fig:d1}
\end{figure}

We denote by $B\defeq P\cap\widetilde{P}$ the bulk of the polygon $P$, which remains untouched by the variation, and by 
\begin{equation*}
D_1\defeq \widetilde{P}\setminus B, \qquad D_n\defeq P\setminus B.
\end{equation*}
In view of the volume constraint we have $|D_1|=|D_n|$. In turn, as each side of $P$ has the same length $\ell$, this implies that
\begin{equation} \label{proof-induction-2}
d_n = - d_1 + O(d_1^2),
\qquad
|D_n|=|D_1|=\ell d_1 + O(d_1^2).
\end{equation}
Furthermore, up to higher order perturbations, the sets $D_1$ and $D_n$ can be approximated by the rectangle $R=[0,\ell]\times[0,d_1]$, in the sense that we can find two rigid motions $T_1,T_n:\R^2\to\R^2$ such that
\begin{equation} \label{proof-induction-3}
|\asymm{\widetilde{D}_i}{R}|=O(d_1^2), \qquad \text{with }\widetilde{D}_i = T_i(D_i).
\end{equation}

We can compute the variation of the nonlocal energy as follows:
\begin{equation} \label{proof-induction-1}
\begin{split}
\big| \nl(\widetilde{P}) - \nl(P) \big|
& = \big| \nl(B\cup D_1) - \nl(B\cup D_n) \big| \\
& = \big| \nl(D_1) + 2\nli(B,D_1) - \nl(D_n) - 2\nli(B,D_n) \big| \\
& \leq \big| \nl(D_1)-\nl(D_n)\big| + 2\big| \nli(B,D_1)-\nli(B,D_n)\big| = (I) + (II).
\end{split}
\end{equation}
We now estimate separately the two terms $(I)$ and $(II)$ on the right-hand side of \eqref{proof-induction-1}.

\medskip\noindent\textit{Estimate on $(I)$.}
Since the nonlocal energy $\nl(D_i)$ is invariant with respect to rigid motions, we can replace the sets $D_i$ by $\widetilde{D}_i$ in $(I)$. Hence, thanks to the estimate \eqref{eq:lipschitz} in Lemma~\ref{lem:potential1}, \eqref{proof-induction-2}, and \eqref{proof-induction-3}, we have
\begin{equation} \label{proof-induction-4}
\begin{split}
(I)
& = \big| \nl(\widetilde{D}_1) - \nl(\widetilde{D}_n)\big| 
\leq 2 c_\alpha|D_1|^{1-\frac{\alpha}{2}} |\asymm{\widetilde{D}_1}{\widetilde{D}_n}| \\
& \leq 2c_\alpha \bigl(\ell d_1 + O(d_1^2)\bigr)^{1-\frac{\alpha}{2}} \Bigl( |\asymm{\widetilde{D}_1}{R}| + |\asymm{\widetilde{D}_n}{R}|\Bigr)
= O(d_1^{3-\frac{\alpha}{2}}).
\end{split}
\end{equation}

\medskip\noindent\textit{Estimate on $(II)$.}
We have to compare the interaction energy of the two small sets $D_1$, $D_n$ with the bulk $B$. The idea here is to transport $D_1$ by a rigid motion to a set $D_1^R$ which maximizes its intersection with $D_n$, and then to exploit the symmetry properties of the polygon $P$ to argue that the interaction of $D_1$ with $B$ is - up to perturbations of order $O(d_1^2)$ - the same as that of $D_1^R$ with $B$, see \eqref{proof-induction-6}.

In order to formalize this strategy, we denote by $H$ the line containing $L_n-|d_n|\nu_n$ and by $H^+$, $H^-$ the two corresponding half-planes (with $H^+$ containing $D_n$). We first consider an isometry $S:\R^2\to\R^2$ such that $S(L_1)=L_n$ and $S(P)=P$, see Remark~\ref{rmk:isometry} (actually, in this particular case the isometry is just the reflection with respect to the bisector of the angle between $L_1$ and $L_n$). With this transformation, the two sets $D_n$ and $S(D_1)$ lie on the two opposite sides of $L_n$. Next, we compose the isometry $S$ with a further translation $T$ by $-|d_n|$ in the direction $\nu_n$ (the exterior normal to $L_n$): we set
\begin{equation*}
D_1^R \defeq T\circ S(D_1), \qquad B^R\defeq T\circ S(B).
\end{equation*}
In this way, one can see that the set $D_1^R$ is isometric to $D_1$ and it overlaps with $D_n$ for most of its volume, in the sense that
\begin{equation} \label{proof-induction-7}
|\asymm{D_1^R}{D_n}| = O(d_1^2).
\end{equation}
We claim that
\begin{equation} \label{proof-induction-6}
\big|\nli(B,D_1)-\nli(B,D_1^R)\big| = O(d_1^2).
\end{equation}
Assuming that the claim \eqref{proof-induction-6} is true, we obtain the desired estimate as follows:
\begin{equation} \label{proof-induction-5}
\begin{split}
(II)
& = 2\big| \nli(B,D_1) - \nli(B,D_n) \big| \xupref{proof-induction-6}{=} 2\big| \nli(B,D_1^R) - \nli(B,D_n) \big| + O(d_1^2) \\
& \leq 2 \nli(B,D_1^R\setminus D_n) + 2 \nli(B, D_n\setminus D_1^R) + O(d_1^2) \\
& \xupref{eq:nlbound2}{\leq} 2c_\alpha |B|^{1-\frac{\alpha}{2}} |\asymm{D_1^R}{D_n}| + O(d_1^2)
\xupref{proof-induction-7}{=} O(d_1^2).
\end{split}
\end{equation}

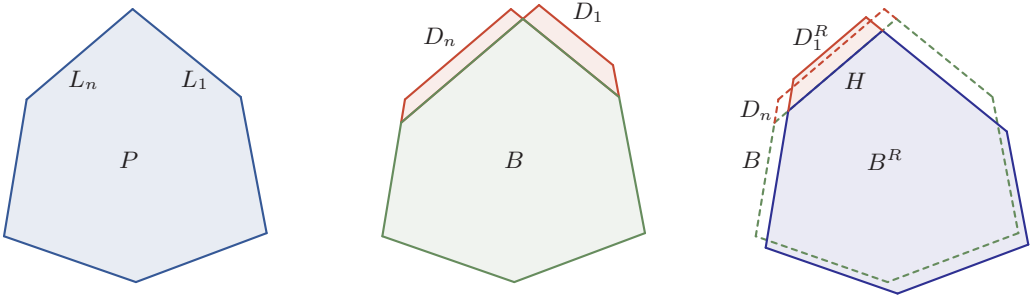
\begin{figure}
\definecolor{qqttcc}{rgb}{0,0.2,0.8}
\definecolor{qqqqff}{rgb}{0,0,1}
\definecolor{ttzzqq}{rgb}{0.2,0.6,0}
\definecolor{ccqqqq}{rgb}{0.8,0,0}
\begin{tikzpicture}[scale=0.65,line cap=round,line join=round,>=triangle 45,x=1cm,y=1cm]
\clip(-19,-0.5) rectangle (5,7);
\fill[line width=1pt,color=qqttcc,fill=qqttcc,fill opacity=0.11] (-17.323574789682244,0.9352393962161015) -- (-16.865935110040915,3.7256648307329048) -- (-14.721711881439338,5.569089613008011) -- (-12.533952407602214,3.777549307400443) -- (-12.009611330486969,0.9988851289092422) -- (-14.65501048396542,0) -- cycle;
\draw [line width=0.8pt,color=qqttcc] (-17.323574789682244,0.9352393962161015)-- (-16.865935110040915,3.7256648307329048);
\draw [line width=0.8pt,color=qqttcc] (-16.865935110040915,3.7256648307329048)-- (-14.721711881439338,5.569089613008011);
\draw [line width=0.8pt,color=qqttcc] (-14.721711881439338,5.569089613008011)-- (-12.533952407602214,3.777549307400443);
\draw [line width=0.8pt,color=qqttcc] (-12.533952407602214,3.777549307400443)-- (-12.009611330486969,0.9988851289092422);
\draw [line width=0.8pt,color=qqttcc] (-12.009611330486969,0.9988851289092422)-- (-14.65501048396542,0);
\draw [line width=0.8pt,color=qqttcc] (-14.65501048396542,0)-- (-17.323574789682244,0.9352393962161015);
\fill[line width=1pt,color=ccqqqq,fill=ccqqqq,fill opacity=0.1] (-9.288542019795122,3.252398286700058) -- (-9.210924626075494,3.7256648307329137) -- (-7.06670139747392,5.569089613008014) -- (-6.824514304827471,5.370764366031469) -- cycle;
\fill[line width=1pt,color=ccqqqq,fill=ccqqqq,fill opacity=0.1] (-6.824514304827471,5.370764366031469) -- (-6.4981377399014715,5.651355789588982) -- (-5.001219497960238,4.425540339258268) -- (-4.878941923636794,3.7775493074004443) -- cycle;
\draw [line width=0.8pt,color=ccqqqq] (-9.288542019795122,3.252398286700058)-- (-9.210924626075494,3.7256648307329137);
\draw [line width=0.8pt,color=ccqqqq] (-9.210924626075494,3.7256648307329137)-- (-7.06670139747392,5.569089613008014);
\draw [line width=0.8pt,color=ccqqqq] (-7.06670139747392,5.569089613008014)-- (-6.824514304827471,5.370764366031469);
\draw [line width=0.8pt,color=ccqqqq] (-6.824514304827471,5.370764366031469)-- (-9.288542019795122,3.252398286700058);
\draw [line width=0.8pt,color=ccqqqq] (-6.824514304827471,5.370764366031469)-- (-6.4981377399014715,5.651355789588982);
\draw [line width=0.8pt,color=ccqqqq] (-6.4981377399014715,5.651355789588982)-- (-5.001219497960238,4.425540339258268);
\draw [line width=0.8pt,color=ccqqqq] (-5.001219497960238,4.425540339258268)-- (-4.878941923636794,3.7775493074004443);
\draw [line width=0.8pt,color=ccqqqq] (-4.878941923636794,3.7775493074004443)-- (-6.824514304827471,5.370764366031469);
\fill[line width=1pt,color=ttzzqq,fill=ttzzqq,fill opacity=0.1] (-9.668564305716822,0.9352393962161007) -- (-9.288542019795122,3.252398286700058) -- (-6.824514304827471,5.370764366031469) -- (-4.878941923636794,3.7775493074004443) -- (-4.3546008465215476,0.9988851289092424) -- (-7,0) -- cycle;
\draw [line width=0.8pt,color=ttzzqq] (-9.668564305716822,0.9352393962161007)-- (-9.288542019795122,3.252398286700058);
\draw [line width=0.8pt,color=ttzzqq] (-9.288542019795122,3.252398286700058)-- (-6.824514304827471,5.370764366031469);
\draw [line width=0.8pt,color=ttzzqq] (-6.824514304827471,5.370764366031469)-- (-4.878941923636794,3.7775493074004443);
\draw [line width=0.8pt,color=ttzzqq] (-4.878941923636794,3.7775493074004443)-- (-4.3546008465215476,0.9988851289092424);
\draw [line width=0.8pt,color=ttzzqq] (-4.3546008465215476,0.9988851289092424)-- (-7,0);
\draw [line width=0.8pt,color=ttzzqq] (-7,0)-- (-9.668564305716822,0.9352393962161007);
\fill[line width=0.8pt,color=qqqqff,fill=qqqqff,fill opacity=0.1] (3.4012371426984513,0.7651236878019863) -- (2.965826264681597,3.0725160750761122) -- (0.4517690036130061,5.13125926310587) -- (-1.4550866368554956,3.491903389625657) -- (-1.9127263164968236,0.7014779551088434) -- (0.7558379892199962,-0.23376144110725683) -- cycle;
\fill[line width=1.2pt,color=ccqqqq,fill=ccqqqq,fill opacity=0.1] (0.4517690036130061,5.13125926310587) -- (0.11876569853224145,5.403953245276756) -- (-1.348363985545077,4.142637174093521) -- (-1.4550866368554956,3.491903389625657) -- cycle;
\draw [line width=0.8pt,color=ttzzqq,dash pattern=on 2pt off 2pt] (-2.1136949600267716,0.9352393962161007)-- (-1.7336726741050708,3.252398286700058);
\draw [line width=0.8pt,color=ttzzqq,dash pattern=on 2pt off 2pt] (-1.7336726741050708,3.252398286700058)-- (0.7303550408625809,5.370764366031469);
\draw [line width=0.8pt,color=ttzzqq,dash pattern=on 2pt off 2pt] (0.7303550408625809,5.370764366031469)-- (2.675927422053257,3.7775493074004443);
\draw [line width=0.8pt,color=ttzzqq,dash pattern=on 2pt off 2pt] (2.675927422053257,3.7775493074004443)-- (3.2002684991685033,0.9988851289092424);
\draw [line width=0.8pt,color=ttzzqq,dash pattern=on 2pt off 2pt] (3.2002684991685033,0.9988851289092424)-- (0.5548693456900509,0);
\draw [line width=0.8pt,color=ttzzqq,dash pattern=on 2pt off 2pt] (0.5548693456900509,0)-- (-2.1136949600267716,0.9352393962161007);
\draw [line width=0.8pt,color=qqqqff] (3.4012371426984513,0.7651236878019863)-- (2.965826264681597,3.0725160750761122);
\draw [line width=0.8pt,color=qqqqff] (2.965826264681597,3.0725160750761122)-- (0.4517690036130061,5.13125926310587);
\draw [line width=0.8pt,color=qqqqff] (0.4517690036130061,5.13125926310587)-- (-1.4550866368554956,3.491903389625657);
\draw [line width=0.8pt,color=qqqqff] (-1.4550866368554956,3.491903389625657)-- (-1.9127263164968236,0.7014779551088434);
\draw [line width=0.8pt,color=qqqqff] (-1.9127263164968236,0.7014779551088434)-- (0.7558379892199962,-0.23376144110725683);
\draw [line width=0.8pt,color=qqqqff] (0.7558379892199962,-0.23376144110725683)-- (3.4012371426984513,0.7651236878019863);
\draw [line width=0.8pt,color=ccqqqq,dash pattern=on 2pt off 2pt] (-1.7336726741050708,3.252398286700058)-- (-1.6560552803854431,3.7256648307329137);
\draw [line width=0.8pt,color=ccqqqq,dash pattern=on 2pt off 2pt] (-1.6560552803854431,3.7256648307329137)-- (0.4881679482161312,5.569089613008014);
\draw [line width=0.8pt,color=ccqqqq,dash pattern=on 2pt off 2pt] (0.4881679482161312,5.569089613008014)-- (0.7303550408625809,5.370764366031469);
\draw [line width=0.8pt,color=ccqqqq] (0.4517690036130061,5.13125926310587)-- (0.11876569853224145,5.403953245276756);
\draw [line width=0.8pt,color=ccqqqq] (0.11876569853224145,5.403953245276756)-- (-1.348363985545077,4.142637174093521);
\draw [line width=0.8pt,color=ccqqqq] (-1.348363985545077,4.142637174093521)-- (-1.4550866368554956,3.491903389625657);
{\footnotesize
\node at (-14.8,2.5) {$P$};
\node [below right] at (-16.2,4.5) {$L_n$};
\node [below left] at (-13,4.5) {$L_1$};
\node at (-7,2.5) {$B$};
\node at (-8.5,5) {$D_n$};
\node at (-5.5,5.5) {$D_1$};
\node at (0.5,2.5) {$B^R$};
\node at (-1,5) {$D_1^R$};
\node at (-2.2,2.5) {$B$};
\node at (-2.1,3.5) {$D_n$};
\node at (-0.1,4.1) {$H$};
}
\end{tikzpicture}
\caption{Left: the starting polygon $P$. Center: the variation $\widetilde{P}$ considered in the proof of Lemma~\ref{lem:induction0}: the bulk $B$ is left untouched, $P=B\cup D_n$, $\widetilde{P}=B\cup D_1$. Right: the dashed line represent the polygon $P=B\cup D_n$; the shaded region is the image of the polygon $\widetilde{P}$ by the isometry $T\circ S$, i.e. the union of the two regions $D_1^R$ and $B^R$.} \label{fig:isometry}
\end{figure}

It only remains to prove the claim \eqref{proof-induction-6}. Due to the invariance of the interaction energy we have that
\begin{equation*}
\begin{split}
\big|\nli(B,D_1)-\nli(B,D_1^R)\big|
& = \big| \nli(B^R,D_1^R)-\nli(B,D_1^R) \big| \\
& = \big| \nli(D_1^R,B^R \setminus B)-\nli(D_1^R,B \setminus B^R) \big| \\
& \leq \nli(D_1^R,\asymm{B^R}{B}) .
\end{split}
\end{equation*}
The estimate of the interaction energy between $D_1^R$ and $\asymm{B^R}{B}$ will follow from the geometric structure of these sets, see Figure~\ref{fig:isometry}. Observe first that the two sets $D_1^R$ and $\asymm{B^R}{B}$ lie on the two opposite sides of the line $H$ ($D_1^R\subset H^+$, $\asymm{B^R}{B}\subset H^-$). Furthermore, the symmetric difference of $B^R$ and $B$ is contained in a neighbourhood of size $2|d_n|$ of the boundary of the original polygon $P$: more precisely,
\begin{equation*}
\asymm{B^R}{B} \subset U\defeq \bigl\{ x\in\R^2\,:\, \mathrm{dist}_P(x,\partial P)<2|d_n| \bigr\} \cap H^-
\end{equation*}
(see \eqref{eq:distance}).
In turn, the set $U$ can be decomposed as disjoint union of $(n-1)$ polygons $U=\bigcup_{i=1}^{n-1} U_i$, each one corresponding approximately to a rectangle of base $\ell$ and height $2|d_n|$ around the side $L_i$. Hence
\begin{equation*}
\nli(D_1^R,\asymm{B^R}{B}) \leq \nli(D_1^R,U) = \sum_{i=1}^{n-1} \nli(D_1^R,U_i) \leq (n-1) \nli(D_1^R,\bar{U}),
\end{equation*}
where $\bar{U}$ is one of the two polygons $U_i$ adjacent to the side $L_n$, for which the interaction with $D_1^R$ is the largest.

As before, up to higher order perturbations, we can approximate the set $D_1^R$ by a rectangle $R_1$ isometric $[0,\ell] \times [0,d_1]$, and $\bar{U}$ by a rectangle $R_2$ isometric to $[0,\ell] \times [0,2d_1]$. The two rectangles are touching each other at one vertex with an angle $\theta\in(0,\pi)$ which corresponds to the internal angle between two adjacent sides of the polygon $P$ (see Figure~\ref{fig:rectangles}). Such approximation is justified since $|\asymm{D_1^R}{R_1}|=O(d_1^2)$, $|\asymm{\bar{U}}{R_2}|=O(d_1^2)$ and hence in view of Lemma~\ref{lem:potential1}
\begin{equation*}
	\nli(D_1^R,\bar{U}) = \nli(R_1,R_2) + O(d_1^2).
\end{equation*}
\begin{figure}
\begin{tikzpicture}[scale=0.7,line cap=round,line join=round,>=triangle 45,x=1cm,y=1cm]
\clip(-7.7,-2.5) rectangle (7.7,2);
\fill[line width=1.2pt,fill=black,fill opacity=0.1] (-7,1) -- (-7,0) -- (0,0) -- (0,1) -- cycle;
\fill[line width=1.2pt,fill=black,fill opacity=0.1] (0.2829888270374986,0.9591232057311202) -- (6.995288866864412,-1.0200184463897062) -- (6.714220942988072,-1.9797063238622956) -- (0,0) -- cycle;
\draw [line width=1.2pt] (-7,1)-- (-7,0);
\draw [line width=1.2pt] (-7,0)-- (0,0);
\draw [line width=1.2pt] (0,0)-- (0,1);
\draw [line width=1.2pt] (0,1)-- (-7,1);
\draw [line width=1.2pt] (0.2829888270374986,0.9591232057311202)-- (6.995288866864412,-1.0200184463897062);
\draw [line width=1.2pt] (6.995288866864412,-1.0200184463897062)-- (6.714220942988072,-1.9797063238622956);
\draw [line width=1.2pt] (6.714220942988072,-1.9797063238622956)-- (0,0);
\draw [line width=1.2pt] (0,0)-- (0.2829888270374986,0.9591232057311202);
\draw [shift={(0,0)},line width=1.2pt]  plot[domain=3.141592653589793:5.996457449924124,variable=\t]({1*0.7490909090909108*cos(\t r)+0*0.7490909090909108*sin(\t r)},{0*0.7490909090909108*cos(\t r)+1*0.7490909090909108*sin(\t r)});
\draw [line width=0.8pt,dash pattern=on 2pt off 2pt] (3.5330964115688666,-2.28351459204545)-- (4,-0.7);
\draw [line width=0.8pt,dash pattern=on 2pt off 2pt] (3.5330964115688666,-2.28351459204545)-- (-0.33675733645832934,-1.1424776791341291);
\draw [line width=0.8pt,dash pattern=on 2pt off 2pt] (-0.33675733645832934,-1.1424776791341291)-- (0,0);
\draw [line width=0.8pt,dash pattern=on 2pt off 2pt] (-5,-1)-- (-5,0.45);
\draw [line width=0.8pt,dash pattern=on 2pt off 2pt] (-7,-1)-- (-7,0);
\draw [line width=0.8pt,dash pattern=on 2pt off 2pt] (-7,-1)-- (-5,-1);
\begin{scriptsize}
\draw [fill=black] (4,-0.7) circle (1.5pt);
\draw [fill=black] (-5,0.45) circle (1.5pt);
\end{scriptsize}
\node [below] at (-6,-1) {{\footnotesize $x_1$}};
\node at (1.45,-2) {{\footnotesize $y_1$}};
\node at (0.2,-0.35) {$\theta$};
\node at (-6.5,0.65) {$R_1$};
\node at (0.5,0.5) {$R_2$};
\node [right] at (-5,0.45) {{\footnotesize $x$}};
\node [right] at (4,-0.7) {{\footnotesize $y$}};
\node [below right] at (-5,0) {\footnotesize $x'$};
\node [below right] at (3.8,-1.2) {\footnotesize $y'$};
\end{tikzpicture}
\caption{Rectangles $R_1$ and $R_2$ of size $[0,\ell]\times[0,d_1]$ approximating the sets $D_1^R$ and $\bar{U}$ in the proof of Lemma~\ref{lem:induction0}.} \label{fig:rectangles}
\end{figure}
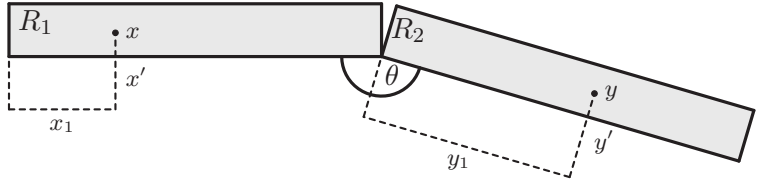
It then remains to estimate $\nli(R_1,R_2)$.
With slight abuse of notation, we identify points $x\in R_1$ and $y\in R_2$ using local coordinates in each rectangle: $x=(x_1,x_2)\in [0,\ell]\times[0,d_1]$, $y=(y_1,y_2)\in [0,\ell]\times[0,2d_1]$. By elementary arguments one can see that the distance between $x$ and $y$ is always larger than or equal to the distance between the corresponding projections $x'=(x_1,0)$, $y'=(y_1,0)$ on the bases of the two rectangles, which in turn obeys the inequality
\begin{equation*}
|x-y| \geq |x'-y'| \geq \sqrt{\frac{1-\cos\theta}{2}} (\ell-x_1+y_1).
\end{equation*}
Therefore, for $\alpha \neq 1$,
\begin{equation*}
\begin{split}
\nli(R_1,R_2) &= \int_{R_1} \int_{R_2} \frac{1}{|x-y|^\alpha} \dd x \dd y \leq C_\theta \int_{R_1} \Bigg( d_1 \int_0^\ell \frac{1}{(\ell-x_1+y_1)^\alpha} \dd y_1 \Bigg) \dd x  \\
				    &= C_\theta \frac{d_1^2 \, \ell^{2-\alpha} (2^{2-\alpha}-2)}{(1-\alpha)(2-\alpha)} = O(d_1^2). 
\end{split}
\end{equation*}
Likewise, for $\alpha=1$, we get that
\begin{equation*}
\nli(R_1,R_2) = \int_{R_1} \int_{R_2} \frac{1}{|x-y|^\alpha} \dd x \dd y \leq C_\theta d_1^2 \, 2\ell \log 2  = O(d_1^2).
\end{equation*}
This proves the estimate \eqref{proof-induction-6}.

\medskip\noindent\textit{Conclusion.}
By inserting \eqref{proof-induction-4} and \eqref{proof-induction-5} into \eqref{proof-induction-1} we obtain
\begin{equation*}
\big| \nl(\widetilde{P}) - \nl(P) \big| = O(d_1^{3-\frac{\alpha}{2}}) + O(d_1^2).
\end{equation*}
As $\alpha<2$, the previous estimate implies the conclusion of the lemma (in the case $d_1>0$).
The case $d_1<0$ follows after relabeling the sets $D_1$ and $D_n$ in the estimates above.
\end{proof}

\begin{lemma} \label{lem:inductionk}
For every $k=2,\ldots,n-1$ there exist $\e_k>0$, $c_k>0$ such that for every choice of $d_1,\ldots,d_{k}$ such that $\sup_{1\leq i \leq k}|d_i|<\e_k$ one has
\begin{equation} \label{eq:inductionk}
\big| V(d_1,\ldots,d_k,0,\ldots,0) - V(d_1,\ldots,d_{k-1},0,\ldots,0) \big| \leq c_k\sup_{1\leq i\leq k}|d_i|^2 .
\end{equation}
\end{lemma}

\begin{proof}
The proof of this lemma follows using estimates almost identical to those in the proof of Lemma~\ref{lem:induction0}. We start by introducing some notation. Let
\begin{align*}
&\mathbf{d}_{k-1} \defeq (d_1,\ldots,d_{k-1},0,\ldots,0)\in\R^{n-1}, &  &P_{k-1}\defeq P(\mathbf{d}_{k-1},f(\mathbf{d}_k-1)), \\
&\mathbf{d}_{k} \defeq (d_1,\ldots,d_{k-1},d_k,0,\ldots,0)\in\R^{n-1},  & &P_{k}\defeq P(\mathbf{d}_{k},f(\mathbf{d}_k)),
\end{align*}
where $f$ is the function defined in \eqref{eq:volumeadj}: the polygon $P_{k-1}$ is obtained from $P$ by translating the first $k-1$ sides by $d_1,\ldots,d_{k-1}$, with the side $L_n$ shifted by $f(\mathbf{d}_{k-1})$ in order to restore the volume constraint; the polygon $P_k$ is obtained from $P_{k-1}$ by translating the side $L_k$ by $d_k$, and the side $L_n$ is shifted by
\begin{equation*}
d_n \defeq f(\mathbf{d}_k)-f(\mathbf{d}_{k-1}).
\end{equation*}
Define also $B \defeq P_{k-1} \cap P_k$, $D_k \defeq P_k \setminus B$, and $D_n \defeq P_{k-1} \setminus B$.
Then, as in \eqref{proof-induction-1}, we get
\begin{equation} \label{nonlocal-diff-k}
\big| \nl(P_k) - \nl(P_{k-1}) \big|  \leq \big| \nl(D_k)-\nl(D_n)\big| + 2\big| \nli(B,D_k)-\nli(B,D_n)\big| = (I) + (II).
\end{equation}
To estimate $(I)$, let $\delta_k \defeq \sup_{1\leq i \leq k}|d_i|$ and note that both $D_k$ and $D_n$ have a side length of $\ell + O(\delta_{k-1})$, $d_n = -d_k + O(d_k^2)$, and $|D_k| = |D_n| = \ell |d_k| + O(\delta_k^2)$. Moreover, as in the proof of the previous lemma, they can be approximated by a rectangle $R$ of side lengths $\ell$ and $|d_k|$ in such a way that 
\begin{equation*}
|\asymm{\widetilde{D}_i}{R}| = O(\delta_k^2), \qquad \text{with } \widetilde{D}_i = T_i(D_i) \text{ for } i=k, n
\end{equation*}
for some rigid motions $T_k, T_n \colon \R^2 \to \R^2$. This, in turn, implies that $|\asymm{\widetilde{D}_k}{\widetilde{D}_n}|=O(\delta_k^2)$, hence, as in \eqref{proof-induction-4}, we obtain that
\begin{equation*}
(I) = \big| \nl(D_k) - \nl(D_{k-1}) \big| = O(\delta_k^2).
\end{equation*}

Now let $S$ be the isometry such that $L_n = S(L_k)$ (see Remark~\ref{rmk:isometry}). Note that $B$ is a polygon with $n$ sides parallel to those of $P$ and each side is of length $\ell + O(\delta_{k})$. As in the estimate of the second term in the proof of Lemma~\ref{lem:induction0}, let $T$ be the translation by $f(\mathbf{d}_k)$ in the direction $\nu_n$ and set $D_k^R \defeq T \circ S(D_k)$ and $B^R \defeq T \circ S(B)$. Hence, $|\asymm{D_k^R}{D_n}| = O(\delta_k^2)$. An important difference here is that the polygon $P_{k-1}$ is not necessarily in the class $\pol$ (and invariant under $S$); however, it is a small perturbation of the polygon $P\in\pol$. Therefore $|\asymm{B}{P}|=|\asymm{S(B)}{P}|=O(\delta_{k})$. Since $|\asymm{T \circ S(B)}{S(B)}| \leq \diam(S(B))|f(\mathbf{d}_k)| = O(\delta_k)$, we get that $|\asymm{B^R}{B}|=O(\delta_k)$. This yields a control on the interaction with the bulk as in \eqref{proof-induction-6}. Then by repeating the estimate \eqref{proof-induction-5} we obtain
\begin{equation*}
(II) = 2\big| \nli(B,D_k)-\nli(B,D_n)\big| = O(\delta_k^2),
\end{equation*}
and conclude that \eqref{eq:inductionk} holds.
\end{proof}

By combining Lemma~\ref{lem:induction0} and Lemma~\ref{lem:inductionk} we can now give the proof of the Main Lemma in Section~\ref{sec:main} by an iteration argument.

\begin{proof}[Proof of Main Lemma]
We set
\begin{equation*}
\e_0\defeq\min_{1\leq i\leq n-1}\e_i\,, \qquad c_0\defeq\sum_{i=1}^{n-1} c_i
\end{equation*}
(where $\e_i$ and $c_i$ are given by Lemma~\ref{lem:induction0} and Lemma~\ref{lem:inductionk}).
For every $P(\mathbf{d})\in\FZ(P,\e_0)$, with $\mathbf{d}=(d_1,\ldots,d_n)$, we then find
\begin{align*}
\big| \nl(P(\mathbf{d}))-\nl(P) \big|
& = \big| V(d_1,\ldots,d_{n-1}) - V(\mathbf{0}) \big| \\
& \leq \big| V(d_1,\ldots,d_{n-1}) - V(d_1,\ldots,d_{n-2},0) \big| + \big| V(d_1,\ldots,d_{n-2},0) - V(\mathbf{0}) \big| \\
& \xupref{eq:inductionk}{\leq} c_{n-1} \biggl(\sup_{1\leq i\leq n-1}|d_i|^2\biggr) + \big| V(d_1,\ldots,d_{n-2},0) - V(\mathbf{0}) \big| \,,
\end{align*}
and therefore by iteration
\begin{align*}
\big| \nl(P(\mathbf{d}))-\nl(P) \big|
& \leq \sum_{i=2}^{n-1}c_{i} \biggl( \sup_{1\leq i\leq n-1}|d_i|^2 \biggr) + \big| V(d_1,0,\ldots,0) - V(\mathbf{0}) \big| \\
& \xupref{eq:induction0}{\leq} c_0 \biggl( \sup_{1\leq i\leq n-1}|d_i|^2 \biggr)  \leq c_0 |\mathbf{d}|_\infty^2 .
\end{align*}
Finally, by observing that (by reducing $\e_0$ if necessary)
\begin{equation*}
|\mathbf{d}|_\infty \leq C |\asymm{P}{P(\mathbf{d})}|
\end{equation*}
for a constant $C>0$ depending only on the initial polygon $P$, we conclude that the estimate \eqref{eq:quadratic} holds with a possibly larger constant $c_0$.
\end{proof}


\section{Proofs of the results in Section~\ref{sec:main}}\label{sec:proofs}

In this section we collect the arguments for the proof of the main results of the paper, stated in Section~\ref{sec:main}.

\begin{proof}[Proof of Theorem~\ref{thm:main1}]
It is proved in \cite[Theorem~3.1]{ChNeuTo20} that there exists $\gamma_1>0$ such that for every $\gamma<\gamma_1$ the minimum problem \eqref{eq:min} has a solution $E_\gamma$. Up to translations, we can assume that $\min_{y\in\R^2}|\asymm{E_\gamma}{(y+P)}|=\asymm{E_\gamma}{P}$. By combining the minimality of $E_\gamma$ with the quantitative Wulff inequality \eqref{eq:quantisop} we obtain
\beqn \label{proofthm1-1}
\bar{c}\,|\asymm{E_\gamma}{P}|^2 \leq \per_\psi(E_\gamma) - \per_\psi(P) \leq \gamma\bigl( \nl(P)-\nl(E_\gamma)\bigr) \leq 2 \gamma c_{\alpha,2}|\asymm{E_\gamma}{P}|,
\eeqn
where in the last inequality we used the Lipschitz continuity of the nonlocal energy with respect to the $L^1$-norm, see \eqref{eq:lipschitz}. The estimate \eqref{proofthm1-1} implies in particular that $E_\gamma\to P$ in $L^1$ as $\gamma\to0$.

Next, we observe that, by a similar estimate using the minimality of $E_\gamma$ and \eqref{eq:lipschitz}, for every set of finite perimeter $F\subset\R^2$ with $|F|=1$ we have
\beqn \label{proofthm1-2}
\per_\psi(E_\gamma) \leq \per_\psi(F) + 2\gamma c_{\alpha,2} |\asymm{E_\gamma}{F}|,
\eeqn
that is, $E_\gamma$ is a $(2\gamma c_{\alpha,2})$-minimizer of the anisotropic perimeter $\per_\psi$. Hence, thanks to \cite[Theorem~7]{FigalliMaggiARMA}, if $\gamma$ is sufficiently small the minimizer $E_\gamma$ is a convex polygon whose sides are parallel to those of $P$, and is uniformly close to $P$ by \eqref{proofthm1-1}: more precisely, there exists $\gamma_2\in(0,\gamma_1)$ such that for all $\gamma<\gamma_2$ we have $E_\gamma\in\FZ(P,\e_0)$, where $\e_0>0$ is the constant provided by the Main Lemma.

We can now conclude the proof thanks to the quadratic bound \eqref{eq:quadratic}: indeed, combining this estimate with the minimality of $E_\gamma$ and the quantitative Wulff inequality \eqref{eq:quantisop} we obtain
\beqn \label{proofthm1-3}
\bar{c}\,|\asymm{E_\gamma}{P}|^2 \leq \per_\psi(E_\gamma) - \per_\psi(P) \leq \gamma\bigl( \nl(P)-\nl(E_\gamma)\bigr) \leq \gamma c_0|\asymm{E_\gamma}{P}|^2,
\eeqn
which implies that $|\asymm{E_\gamma}{P}|=0$ provided that $\gamma<\frac{\bar{c}}{c_0}$.
\end{proof}

We now prove that a critical point of the nonlocal energy $\nl(\cdot)$, with respect to perturbations in the class $\FZ$, satisfies the first order conditions \eqref{eq:critical}.

\begin{proof}[Derivation of conditions \eqref{eq:critical}]
We consider a convex polygon $P$ with $n$ sides $\{L_i\}_{i=1}^n$ and corresponding side lengths $\{\ell_i\}_{i=1}^n$, and its perturbations in the class $\FZ(P,\e)$. With the notation of Section~\ref{sec:quadratic}, we impose the conditions
\begin{equation} \label{proof-crit-1}
\frac{\partial V}{\partial d_i}(0,\ldots,0)=0 \qquad\text{for all }i=1,\ldots,n-1,
\end{equation}
where $V$ is the function introduced in \eqref{eq:V}. Let us compute the derivative $\frac{\partial V}{\partial d_1}$, corresponding to a shifting of the side $L_1$ by a small quantity $d_1$ (at the same time the side $L_n$ is shifted by $d_n$ to adjust the volume constraint): we refer once again to Figure~\ref{fig:d1} and to the notation introduced in Lemma~\ref{lem:induction0}.

Since we are not assuming that the sides of the polygon have equal length, the condition \eqref{proof-induction-2} has to be replaced by
\begin{equation} \label{proof-crit-2}
d_n = -\frac{\ell_1}{\ell_n}d_1 + O(d_1^2).
\end{equation}
We have for $d_1>0$ sufficiently small
\begin{equation} \label{proof-crit-3}
\frac{1}{d_1}\Bigl( V(d_1,0,\ldots,0) - V(0,\ldots,0)\Bigr)
= \frac{1}{d_1}\Bigl( \nl(D_1) +2\nli(D_1,B) - \nl(D_n) - 2\nli(D_n,B) \Bigr).
\end{equation}
Since $|D_1|=|D_n|=\ell_1d_1+O(d_1^2)$, it follows from Lemma~\ref{lem:potential1} that $\nl(D_1), \nl(D_n)=o(d_1)$ as $d_1\to0$, so that these two terms are asymptotically negligible in \eqref{proof-crit-3}. Furthermore, denoting by $R_1$ the rectangle with base $L_1$ and height $d_1$ (on the exterior of the polygon $P$), we have the estimates
\begin{equation*}
|\asymm{B}{P}| = |D_n|=\ell_1d_1+O(d_1^2), \qquad |\asymm{D_1}{R_1}|=O(d_1^2),
\end{equation*}
from which it follows, again by Lemma~\ref{lem:potential1}, that $\nli(D_1,B)=\nli(R_1,P) + o(d_1)$. Therefore
\begin{equation*}
\begin{split}
\lim_{d_1\to0^+}\frac{\nli(D_1,B)}{d_1}
& =\lim_{d_1\to0^+}\frac{\nli(R_1,P)}{d_1} \\
& = \lim_{d_1\to0^+}\frac{1}{d_1}\int_{0}^{d_1}\mathrm{d}t \int_{L_1}\mathrm{d}\Hone(x) \int_{P}\frac{\dd y}{|y-(x+t\nu_1)|^\alpha}\\
& = \int_{L_1}\int_P\frac{1}{|x-y|^\alpha}\dd y\dd\Hone(x).
\end{split}
\end{equation*}
Arguing similarly for $D_n$, we let $R_n$ be the rectangle of base $L_n$ and height $|d_n|$ (inside the polygon $P$), and taking into account \eqref{proof-crit-2} we find
\begin{equation*}
\begin{split}
\lim_{d_1\to0^+}\frac{\nli(D_n,B)}{d_1}
& =\lim_{d_1\to0^+}\frac{\nli(R_n,P)}{d_1} \\
& = \lim_{d_1\to0^+}\frac{1}{d_1}\int_{0}^{\frac{\ell_1}{\ell_n}d_1+o(d_1)}\mathrm{d}t \int_{L_n}\mathrm{d}\Hone(x) \int_{P}\frac{\dd y}{|y-(x-t\nu_n)|^\alpha}\\
& = \frac{\ell_1}{\ell_n}\int_{L_n}\int_P\frac{1}{|x-y|^\alpha}\dd y\dd\Hone(x).
\end{split}
\end{equation*}
Therefore inserting all the previous estimates into \eqref{proof-crit-3} we conclude that
\begin{equation*}
\begin{split}
\lim_{d_1\to0^+} \frac{1}{d_1}\Bigl( & V(d_1,0,\ldots,0) - V(0,\ldots,0)\Bigr)\\
& = 2\biggl(\int_{L_1}\int_P\frac{1}{|x-y|^\alpha}\dd y\dd\Hone(x) - \frac{\ell_1}{\ell_n}\int_{L_n}\int_P\frac{1}{|x-y|^\alpha}\dd y \dd\Hone(x)\biggr).
\end{split}
\end{equation*}
The computation for $d_1$ negative is similar and yields that the right-hand side in the previous identity is exactly the partial derivative $\frac{\partial V}{\partial d_1}(0,\ldots,0)$. By imposing \eqref{proof-crit-1} we then obtain \eqref{eq:critical} for $i=1$ and $j=n$; the other conditions follow by considering the analogous variations for the other sides.
\end{proof}

We next give the proof of Theorem~\ref{thm:main2}, following the argument in \cite[Theorem~14]{FigalliMaggiARMA}.

\begin{proof}[Proof of Theorem~\ref{thm:main2}]
We first observe that there exists $\omega>0$, depending on $\alpha$, $\gamma$, $\delta_0$, and $E$, such that $E$ is an $\omega$-minimizer of the anisotropic perimeter $\per_\psi$. Indeed, for every $F\subset\R^2$ with finite perimeter such that $|F|=|E|$ and $|\asymm{E}{F}|<\delta_0$ we have by local minimality of $E$
\begin{equation*}
\per_\psi(E) \leq \per_\psi(F) + \gamma\big|\nl(F)-\nl(E) \big| \leq \per_\psi(F) + 2c_{\alpha,2}\gamma |E|^{1-\frac{\alpha}{2}}|\asymm{E}{F}|,
\end{equation*}
where the last inequality follows from Lemma~\ref{lem:potential1}. Moreover for competitors $F$ such that $|F|=|E|$ and $|\asymm{E}{F}|\geq\delta_0$ we have trivially
\begin{equation*}
\per_\psi(E) \leq \per_\psi(F) + \frac{\per_\psi(E)}{\delta_0}|\asymm{E}{F}|.
\end{equation*}
The $\omega$-minimality of $E$ then follows by taking $\omega\defeq\max\{ 2c_{\alpha,2}\gamma |E|^{1-\frac{\alpha}{2}},\per_\psi(E)/\delta_0 \}$.
We can therefore apply \cite[Theorem~3]{FigalliMaggiARMA} to deduce that $\Hone(\partial E\setminus\partial^*E)=0$ and $\partial E$ is differentiable at every point of $\partial^*E$.

These regularity properties are sufficient to repeat the construction of the perturbation considered in \cite[Theorem~14]{FigalliMaggiARMA}. It is shown that, given $\bar{x}\in\partial^*E$ such that $\nu_E(x)\notin\{\nu_i\}_{i=1}^n$, there exists $\sigma_0(\bar{x})>0$ with the following property: for every $\sigma\in(0,\sigma_0(\bar{x}))$ one can construct two sets
\begin{equation*}
R_+^\sigma(\bar{x}) \subset \R^2\setminus E, \qquad R_-^\sigma(\bar{x}) \subset E, \qquad\text{with }|R_+^\sigma(\bar{x})|=|R_-^\sigma(\bar{x})|=\sigma,
\end{equation*}
such that
\begin{equation} \label{proofthm2-1}
\per_\psi(E\cup R_+^\sigma(\bar{x})) \leq \per_\psi(E),
\qquad
\per_\psi(E\setminus R_-^\sigma(\bar{x})) \leq \per_\psi(E).
\end{equation}
Moreover, the two sets are localized around $\bar{x}$ in the sense that there exists $\rho(\sigma)>0$, with $\rho(\sigma)\to 0$ as $\sigma \to 0$, such that $R_{\pm}^\sigma(\bar{x})$ are contained in the ball $B_{\rho(\sigma)}(\bar{x})$.

Let now $x_1,x_2\in\partial^*E$ be such that $\nu_E(x_i)\notin\{\nu_i\}_{i=1}^n$. The proof will be achieved if we show that $v_E(x_1)=v_E(x_2)$.
For $0<\sigma<\sigma_0\defeq\min\{\sigma_0(x_1),\sigma_0(x_2)\}$ we consider the two sets $R_+^\sigma\defeq R_+^\sigma(x_1)$, $R_-^\sigma\defeq R_-^\sigma(x_2)$. By reducing $\sigma_0$ if necessary we can guarantee that $R_+^\sigma\cap R_-^\sigma=\emptyset$. We then consider the competitor
\begin{equation*}
F^\sigma\defeq (E\cup R_+^\sigma)\setminus R_-^\sigma, \qquad |F^\sigma|=|E|.
\end{equation*}
If $\sigma_0$ is sufficiently small we have $|\asymm{F^\sigma}{E}|<\delta_0$ and therefore by local minimality of $E$
\begin{equation*}
\per_\psi(E) + \gamma\nl(E) \leq \per_\psi(F^\sigma) + \gamma\nl(F^\sigma) \leq \per_\psi(E) + \gamma\nl(F^\sigma).
\end{equation*}
The last inequality follows from \eqref{proofthm2-1} taking into account that the perturbations $R_{\pm}^\sigma$ are localized in two disjoint balls $B_{\rho(\sigma)}(x_1)$, $B_{\rho(\sigma)}(x_2)$. Hence
\begin{align} \label{proofthm2-2}
\nl(E)
&\leq \nl(F^\sigma) 
= \nl(E) + 2\nli(E,R^\sigma_+) -2\nli(E,R^\sigma_-) + \nl(R^\sigma_+) + \nl(R^\sigma_-) - 2\nli(R^\sigma_+,R^\sigma_-).
\end{align}
In view of Lemma~\ref{lem:potential1} we have
\begin{equation*}
\big| \nl(R^\sigma_+) + \nl(R^\sigma_-) - 2\nli(R^\sigma_+,R^\sigma_-) \big| \leq 4c_{\alpha,2}\, \sigma^{2-\frac{\alpha}{2}}.
\end{equation*}
Therefore dividing \eqref{proofthm2-2} by $\sigma$ and letting $\sigma\to0^+$ we obtain
\begin{equation*}
\lim_{\sigma\to0^+} \frac{\nli(E,R^\sigma_-)}{\sigma} \leq \lim_{\sigma\to0^+}\frac{\nli(E,R^\sigma_+)}{\sigma},
\end{equation*}
or equivalently
\begin{equation*}
\lim_{\sigma\to0^+} \frac{1}{|R^\sigma_-|}\int_{R^\sigma_-} v_E(x)\dd x  \leq \lim_{\sigma\to0^+} \frac{1}{|R^\sigma_+|}\int_{R^\sigma_+} v_E(x)\dd x .
\end{equation*}
By continuity of the potential we conclude that $v_E(x_2)\leq v_E(x_1)$. The opposite inequality follows by inverting the roles of $x_1$ and $x_2$.
\end{proof}


\section{Higher-dimensional case}\label{sec:3d}

This section is devoted to proving the analogous of Theorem~\ref{thm:main1} in higher dimension.
In the following we will fix the dimension $d\geq 2$.
We first introduce the class of polytopes that we consider.

\begin{definition}\label{def:class_gen}
For $n\geq d+1$, we let $\pol(d)$ to be the class of convex polytopes $P\subset\R^d$ with $n$ faces $F_1, \dots, F_n$  and unit volume $|P|=1$ such that for every $i, j\in\{1,\dots,n\}$ there exists an isometry $S_{ij}:\R^d\to\R^d$ such that $S_{ij}(F_i)=F_j$ and $S_{ij}(P)=P$.
\end{definition}

\begin{remark}
This is the straightforward generalization of the class of polygons considered in the two dimensional case (see Definition~\ref{def:polygons} and Remark~\ref{rmk:isometry}).
Notice that regular polytopes belong to $\pol(d)$.
\end{remark}

We are now in position to state the main result of this section.

\begin{theorem}[Minimality of polytopes in $\pol(d)$] \label{thm:main3}
Let $P\in\pol(d)$ and let $\psi$ be a surface energy density whose Wulff shape is $P$.
Then there exists $\bar\gamma>0$ such that for all $\gamma<\bar{\gamma}$ the polytope $P$ is the unique (up to translations) solution to \eqref{eq:min}.
\end{theorem}

Since minimizers of \eqref{eq:min} are also $\omega$-minimizers of the anisotropic perimeter for some $\omega$ proportional to $\gamma$, \cite[Theorem~1.1]{FigZha} allows to consider only competitors in the class $\FZ(P,\e)$ (see Definition~\ref{def:fzclass}, whose generalization to dimension $d>2$ is straightforward). Therefore Theorem~\ref{thm:main3} follows by using the same strategy that we used to prove Theorem~\ref{thm:main1}, once we get the following estimate.

\begin{lemma} \label{lem:general_estimate}
Let $P\in\pol(d)$.
Then there exist $\e_0>0$ and $c_0>0$ such that for every $\widetilde{P}\in\FZ(P,\e_0)$ one has the quadratic estimate
\beqn \label{eq:quadratic2}
\big| \nl(\widetilde{P}) - \nl(P)  \big| \leq c_0 |\asymm{P}{\widetilde{P}}|^2 \,.
\eeqn
\end{lemma}

\begin{proof}
Since the argument is similar to the one used to prove the estimate in the two dimensional case, here we limit ourselves to sketch the main changes in the proof.

\medskip\noindent
\textit{Step 1.} We claim that there exist $\varepsilon_1>0$  and $c_1>0$ such that for every $|d_1|<\varepsilon_1$ it holds
\begin{equation}\label{eq:gen_est_one}
\big| V(d_1,0,\dots,0) - V(0,\dots,0) \big| \leq c_1 |d_1|^2.
\end{equation}
Indeed, let us consider $d_1>0$ sufficiently small and set
\[
\widetilde{P}\defeq P(d_1,0,\dots,0), \quad\quad
B\defeq P\cap \widetilde{P},\quad\quad
D_1\defeq\widetilde{P}\setminus B,\quad\quad
D_n\defeq P\setminus\widetilde{P}.
\]
The volume constraint $|P|=|\widetilde{P}|$ together with the fact that $\mathcal{H}^{d-1}(F_1) = \mathcal{H}^{d-1}(F_n)$ yield
\begin{equation}\label{eq:d1dn_higher}
d_n = -d_1 + O(d_1^2).
\end{equation}
We have that
\begin{equation}\label{eq:first_estiamte_gen}
\big| \mathcal{V}(\widetilde{P}) - \mathcal{V}(P)  \big| \leq \big| \nl(D_1)-\nl(D_n)\big| + 2\big| \nli(B,D_1)-\nli(B,D_n)\big| .
\end{equation}

We start by considering the first term on the right-hand side.
Let $\nu_i\in\mathbb{S}^{d-1}$ be the normal to the face $F_i$ pointing outside $P$, and set
\begin{align*}
R_1 &\defeq \bigl\{x\in \R^d \,:\, x = y+ t \nu_1,\, y\in  F_1,\, t\in  [0,d_1] \bigr\}, \\
R_n &\defeq \bigl\{x\in \R^d \,:\, x = y+ t \nu_n,\, y\in  F_n,\, t\in  [-d_1,0] \bigr\}.
\end{align*}
Notice that, for $i=1, n$ it holds
\begin{equation}\label{eq:estR1Rn_gen}
|\asymm{R_i}{D_i}| = O(d_1^2).
\end{equation}
Using Lemma~\ref{lem:potential1} together with \eqref{eq:d1dn_higher} and \eqref{eq:estR1Rn_gen}, and arguing as in \eqref{proof-induction-4} we get
\begin{align}\label{eq:estD1Dn_gen}
\big| \nl(D_1) - \nl(D_n)\big|  = O(d_1^{3-\frac{\alpha}{2}}).
\end{align}

We next estimate the second term on the right-hand side of \eqref{eq:first_estiamte_gen}. We notice that
\begin{equation}\label{eq:simplification_gen}
\begin{split}
\big| \nli(B,D_1)  - \nli(B,D_n)\big|  &\leq \big| \nli(B,D_1)-\nli(B,R_1)\big|  + \big| \nli(B,D_n)-\nli(B,R_n)\big| \\
&\hspace{2cm}	+ \big| \nli(B,R_1)-\nli(B,R_n)\big|  \\
& = \big| \nli(B,R_1) - \nli(B,R_n)\big| + O(d_1^2).
\end{split}
\end{equation}
We now focus on the first term on the right-hand side of \eqref{eq:simplification_gen}. Let $S$ be the isometry such that $S(F_n)=F_1$ and $S(P)=P$, given by Definition~\ref{def:class_gen}. Consider the composition of $S$ with a translation $T$ by $d_1\nu_1$: the resulting transformation maps $R_n$ onto $R_1$, i.e. $T\circ S(R_n)=R_1$. We also set $\widetilde{B}\defeq T\circ S(B)$. In view of the invariance of $P$ with respect to $S$, it is then possible to estimate
\begin{equation}\label{eq:estBR1_gen}
\begin{split}
\big| \nli(B,R_1) - \nli(B,R_n)\big|
& = \big| \nli(B,R_1)-\nli(\widetilde{B},R_1)\big| \\
& \leq \nli( R_1, \asymm{\widetilde{B}}{B} ) \\
& \leq (n-1)\nli( R_1, U_i ) + O(d_1^2)
\end{split}
\end{equation}
where
\[
U_i := \bigl\{x\in \R^d \,:\,  x = y+ t \nu_i,\,  y\in  F_i,\, t\in  [-2d_1,2d_1] \bigr\}\setminus R_1,
\]
and $i\in{1,\dots,n-1}$ is such that the interaction of $R_1$ with $U_i$ is maximal.
By elementary geometry one can see that there exists $C>0$ depending only on the angle between $\nu_1$ and $\nu_i$ such that for every $x\in R_1$ and $y\in U_i$
\[
| x-y | \geq C |\pi_1(x)-\pi_i(y)|,
\]
where $\pi_1$ and $\pi_i$ are the projections on $F_1$ and $F_i$ respectively. Therefore using Fubini's Theorem
\begin{equation*}
\begin{split}
\nli(R_1,U_i)& = \int_{R_1} \int_{U_i} \frac{1}{| x-y |^\alpha} \dd y \dd x \\
& \leq C \int_{R_1} \int_{U_i} \frac{1}{| \pi_1(x) - \pi_i(y) |^\alpha} \dd y \dd x \\
& = C d_1^2 \int_{F_1} \int_{F_i}  \frac{1}{|x'-y'|^\alpha} \dd \mathcal{H}^{d-1}(y') \dd \mathcal{H}^{d-1}(x').
\end{split}
\end{equation*}
As $\alpha<d$, the last integral is a finite constant depending only on the polygon, hence
\begin{equation}\label{eq:est_angle_gen}
\nli(R_1,U_i) = O(d_1^2).
\end{equation}
Combining \eqref{eq:first_estiamte_gen}, \eqref{eq:estR1Rn_gen}, \eqref{eq:estD1Dn_gen}, \eqref{eq:simplification_gen}, \eqref{eq:estBR1_gen},  and \eqref{eq:est_angle_gen} we get \eqref{eq:gen_est_one}.

\medskip\noindent
\textit{Step 2.} We conclude the proof of the lemma by using the iteration argument of Lemma~\ref{lem:inductionk}. Minor changes in the proof in order to adapt it to the general dimensional case are left to the reader.
\end{proof}


\subsection*{Acknowledgments}
The authors would like to thank Gian Paolo Leonardi and Massimiliano Morini for interesting discussions on the topic of this paper. Part of this work was carried out during IT's visit to the University of Trento.  He gratefully acknowledges the hospitality of the Department of Mathematics.
MB is member of the GNAMPA group of the Istituto Nazionale di Alta Matematica (INdAM).
The research of RC has been supported by grant EP/R013527/2 ``Designer Microstructure via Optimal Transport Theory'' of David Bourne.

\bibliographystyle{IEEEtranS}
\def\url#1{}
\bibliography{references}

\end{document}